\pdfoutput=1 % needed for arXiV submission if use microtype package

\documentclass[12pt]{extarticle}
\usepackage[utf8]{inputenc}
\usepackage[T1]{fontenc}

\usepackage{caption}
\oddsidemargin \evensidemargin
\usepackage{float}
\usepackage{resizegather}
\usepackage{amssymb}
\usepackage{pmboxdraw}
\usepackage{amsmath}
\usepackage{bbm}
\usepackage{amsthm}
\usepackage{amssymb}
\usepackage{dsfont}
\usepackage[dvipsnames]{xcolor}
\usepackage{graphicx}
\usepackage{alphabeta}
\usepackage{tikz,pgfplots}
\pgfplotsset{compat=1.18}
\usepackage{pgfplots}
\usepackage{comment}
\usepackage[margin=0.5cm]{caption}
\usepackage{hyperref}
\usepackage{subcaption}
\usepackage[all]{xy}
\usepackage[greek,english]{babel}
\usepackage[export]{adjustbox}
\usepackage{hyperref}
\usepackage{mathrsfs}  
\hypersetup{
    colorlinks=true,
    linkcolor=myblue,
    filecolor=magenta,      
    urlcolor=myblue,
    citecolor=BurntOrange,
}
\usepackage{enumitem}
\usepackage{booktabs}

\definecolor{myblue}{rgb}{0.25,0.45,0.99}
\definecolor{myorange}{rgb}{0.8500, 0.3250, 0.0980}
\definecolor{myyellow}{rgb}{0.9290, 0.6940, 0.1250}
\definecolor{mypurple}{rgb}{0.4940, 0.1840, 0.5560}
\definecolor{mygreen}{rgb}{0.4660, 0.6740, 0.1880}

\usepackage[activate={true,nocompatibility},final,tracking=true,kerning=true,spacing=true,factor=1100,stretch=10,shrink=10]{microtype}

\let\OLDthebibliography\thebibliography
\renewcommand\thebibliography[1]{
  \OLDthebibliography{#1}
  \setlength{\parskip}{0.4pt}
  \setlength{\itemsep}{3.0pt plus 0.3ex}
}

\usepackage{listings}

\usepackage{algorithm}
\usepackage{algpseudocode}

\newtheorem{theorem}{Theorem}[section]
\newtheorem{theorem*}[theorem]{Theorem*}

\newtheorem{corollary}[theorem]{Corollary}
\newtheorem{proposition}[theorem]{Proposition}

\newtheorem{thm/conj}[theorem]{Theorem/Conjecture}

\theoremstyle{definition}
\newtheorem{examplex}[theorem]{Example}
\newenvironment{example}
{\pushQED{\qed}\examplex}
{\popQED\endexamplex}
\newtheorem{exercisex}{Exercise}
\newenvironment{exercise}
{\pushQED{\qed}\exercisex}
{\popQED\endexercisex}
\newtheorem{definition}[theorem]{Definition}
\newtheorem{remark}[theorem]{Remark}

\theoremstyle{remark}

\newcommand\restr[2]{{
  \left.\kern-\nulldelimiterspace 
  #1
  \vphantom{\big|} 
  \right|_{#2}
  }}

\usepackage{titling} % decrease space above title
\title{\textbf{Positive Geometry of Polytopes and Polypols}
}
\author{Simon Telen}
\date{}

\begin{document}

\maketitle

\begin{abstract}
    \noindent These are lecture notes supporting a minicourse taught at the Summer School in Total Positivity and Quantum Field Theory at CMSA Harvard in June 2025. We give an introduction to positive geometries and their canonical forms. We present the original definition by Arkani-Hamed, Bai and Lam, and a more recent definition suggested by work of Brown and Dupont. We compute canonical forms of convex polytopes and of quasi-regular polypols, which are nonlinear generalizations of polygons in the plane. The text is a collection of known results. It contains many examples and a list of exercises. 
\end{abstract}

\section{Definitions and first examples} \label{sec:1}

Positive Geometry is an emerging field of mathematics with origins in theoretical physics \cite{ranestad2025positive}. It studies the geometry and combinatorics underlying physical observables such as scattering amplitudes and cosmological correlators. Here ``geometry'' is meant in a broad sense, including complex, real and tropical algebraic geometry. The discovery of the amplituhedron by Arkani-Hamed and Trnka marks one of the first milestones of the subject \cite{arkani2014amplituhedron}. This was followed by numerous other geometric structures in particle physics and cosmology, of which we mention correlahedra \cite{eden2017correlahedron}, ABHY associahedra \cite{arkani2018scattering}, surfacehedra \cite{arkani2023all,arkani2023all2}, cosmological polytopes \cite{arkani2017cosmological} and cosmohedra \cite{arkani2024cosmohedra}. Physical principles underlying all these objects cause them to have several common features. This motivates the development of a unifying mathematical framework. The first step was taken in \cite{arkani2017positive}, which gives an axiomatic definition of a \emph{positive geometry}. This definition is recursive and somewhat technical. However, it is easy to illustrate with examples. The recent Hodge-theoretic approach from \cite{brown2025positive} suggests an alternative definition, which is easier to state. A reader with knowledge of mixed Hodge theory may find the latter definition more natural. The two definitions are not quite equivalent, but most examples in these lectures align with both. We start with the classical definition following \cite{arkani2017positive, lam2024invitation}.  

Let $X$ be an irreducible complex projective variety of dimension $d$. Let $X_{\geq 0}$ be a $d$-dimensional closed semi-algebraic subset of the real points $X(\mathbb{R}) \subseteq X$. This assumes that $X$ has ``sufficiently many'' real points. We require $X_{\geq 0}$ to equal the closure of its interior: $X_{\geq 0} = \overline{{\rm int}(X_{\geq 0})}$. Here interior and closure are with respect to the analytic topology on $X(\mathbb{R})$. Moreover, we assume that ${\rm int}(X_{\geq 0})$ is an oriented real manifold. The Zariski closure of the euclidean boundary $\partial X_{\geq 0} = X_{\geq 0} \setminus {\rm int}(X_{\geq 0})$ is a divisor on $X$, denoted by $Y$. Its irreducible decomposition is $Y = Y_1 \cup \cdots \cup Y_r$. The divisor $Y$ is called the \emph{algebraic boundary} of $X_{\geq 0}$.

\begin{example}
Figure \ref{fig:firstexamples} shows two examples. On the left, $X$ is a conic in $\mathbb{P}^2$ given by $x_0^2-x_1^2-x_2^2 = 0$, and $X_{\geq 0}$ is the green curve segment $X \cap \mathbb{P}^2_{\geq 0}$. The black circle containing $X_{\geq 0}$ consists of the real points $X(\mathbb{R})$ of $X$. The algebraic boundary $Y$ consists of two points $(r = 2)$. On the right, $X = \mathbb{P}^2$ and $X_{\geq 0}$ is a semi-algebraic set bounded by two lines and a conic. The euclidean boundary $\partial X_{\geq 0}$ is the bold blue curve. We choose the orientations as indicated in pink. A non-example is given by the semialgebraic set $X_{\geq 0} = \{ x \in \mathbb{R} \, : \, x^2(x-1)(x-2) \leq 0 \}$, viewed as a subset of $X = \mathbb{P}^1$. This violates the requirement $X_{\geq 0} = \overline{{\rm int}(X_{\geq 0})}$. 
    \begin{figure}
        \centering
        \includegraphics[width=0.65\linewidth]{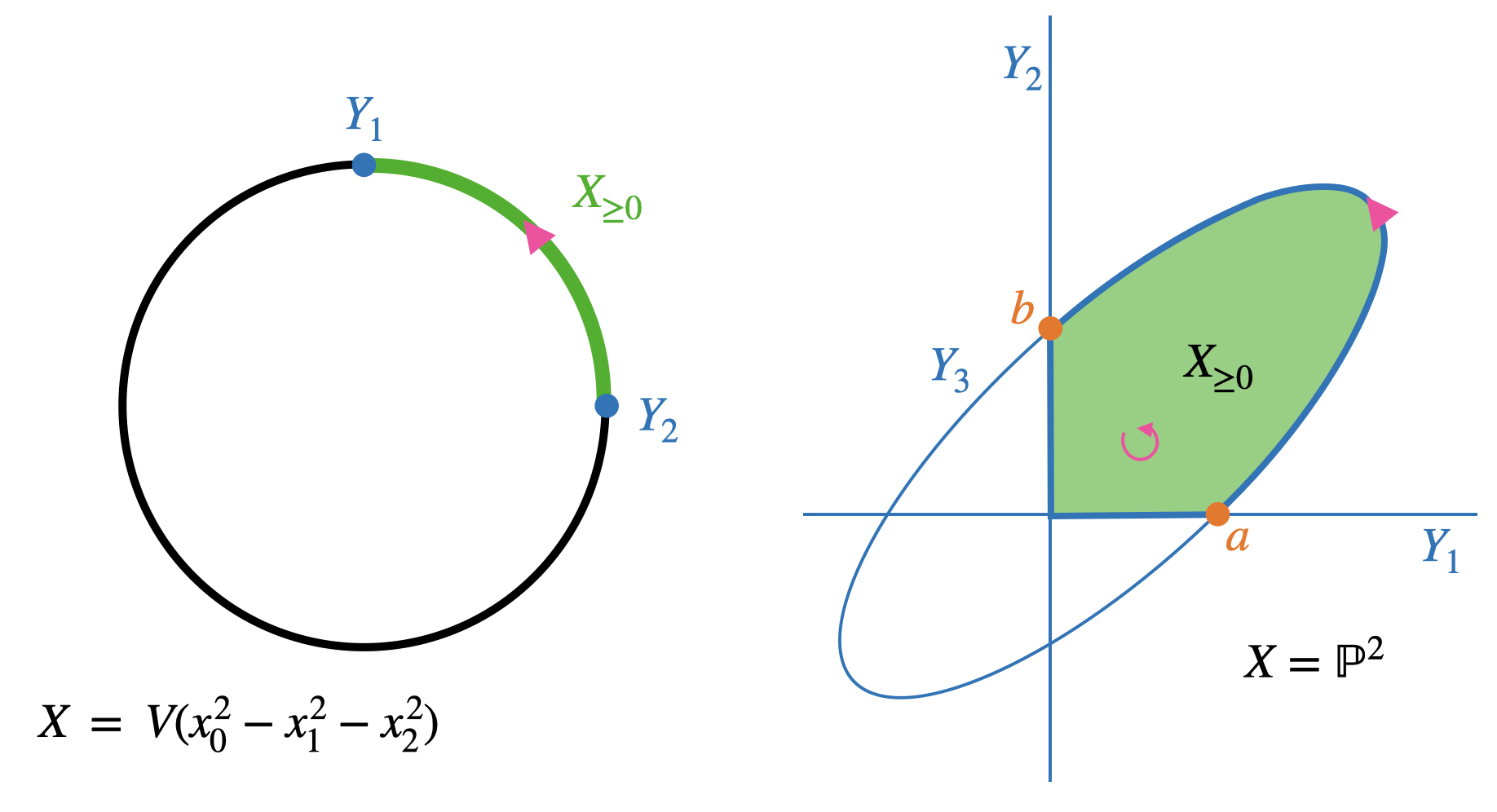}
        \caption{Two positive geometries of dimension $d = 1$ (left) and $d = 2$ (right).}
        \label{fig:firstexamples}
    \end{figure}
\end{example}

The tuple $(X,X_{\geq 0})$ is the data required for a candidate positive geometry in the sense of our first definition (Definition \ref{def:posgeomABL}). That definition requires the ``iterated boundaries'' of $X_{\geq 0}$ to be positive geometries of smaller dimension. To make this precise, for $i = 1, \ldots, r$, let $(Y_i)_{\geq 0} = \overline{{\rm int}(X_{\geq 0} \cap Y_i)}$, where interior and closure are with respect to the analytic topology on $Y_i(\mathbb{R})$, the real points of $Y_i$. In Figure \ref{fig:firstexamples} (right), $(Y_3)_{\geq 0}$ is the curve segment connecting the points $a$ and $b$ along $Y_3$. The orientation of ${\rm int}((Y_3)_{\geq 0})$ is induced by that of ${\rm int}(X_{\geq 0})$. 

Positive geometries are equipped with a differential $d$-form. By a \emph{differential $d$-form on $X$} we mean a differential $d$-form on the smooth manifold $X \setminus {\rm Sing}(X)$, the smooth points of $X$. 
\begin{example}
    A rational $d$-form on $X = \mathbb{P}^d$ is given locally on $\mathbb{C}^d \subset \mathbb{P}^d$ by 
    \begin{equation} \label{eq:g/f} \omega \, = \, \frac{g(x_1,\ldots, x_d)}{f(x_1, \ldots, x_d)} \, {\rm d} x_1 \wedge \cdots \wedge {\rm d} x_d,\end{equation}
    where $f, g \in \mathbb{C}[x_1, \ldots, x_d]$ are coprime and $f \neq 0$. The closure in $\mathbb{P}^d$ of an irreducible component of the affine hypersurface $g = 0$ is called a \emph{zero} of $\omega$. Similarly, the closure of a component of $f=0$ is a \emph{pole} of $\omega$. The \emph{order} of a pole locally given by $f_0 = 0$ is the largest $k$ such that $f$ can be written as $f = f_0^k f_1$ for some polynomial $f_1$. The analogous definition gives the order of a zero. Each rational $d$-form on $\mathbb{P}^d$ has finitely many poles and zeros.

    More generally, if $X$ is smooth and covered by affine spaces (e.g., $X$ is a smooth projective toric variety or a Grassmannian), then $\omega$ is also represented by \eqref{eq:g/f} on an affine chart $\mathbb{C}^d \subseteq X$.  

    If $X$ is any smooth $d$-dimensional algebraic variety, a rational $d$-form $\omega$ is such that on each chart $\varphi: U \rightarrow V \subset \mathbb{C}^d$ of $X$, $\omega$ is given by $h(x) \, {\rm d} x$ where $h: V \rightarrow \mathbb{C}$ is meromorphic and $\varphi^*h: U \rightarrow \mathbb{C}$ is the restriction of a rational function on $X$. The \emph{order of vanishing} of $\omega$ along a prime divisor $Y_i \subset X$ is defined as in \cite[Chapter II, \S 6]{hartshorne2013algebraic}. Such a prime divisor is a pole of order $k$ if the order of vanishing is $-k$, and a zero of order $k$ if the order of vanishing is $k$. In these lectures we mostly work with the concrete definition for $X = \mathbb{P}^d$.
\end{example}
The final ingredient we need to state the first definition is the (Poincar\'e) \emph{residue} of a differential $d$-form. Let $\omega$ be a differential $d$-form on $X$ with a pole of order at most one along a prime divisor $Y_i \subset X$, not contained in the singular locus ${\rm Sing}(X)$. A pole of order one is called a \emph{simple pole}. Locally at a smooth point of $Y_i \setminus {\rm Sing}(X)$ we~have
\begin{equation} \label{eq:decompomega} \omega \, = \,  \eta \wedge \frac{{\rm d} f}{f} + \eta',\end{equation}
where $\eta$ is a $(d-1)$-form and $\eta'$ a $d$-form, both with no poles along $Y_i$, and $f$ is a local coordinate which vanishes to order one along $Y_i$. The residue of $\omega$ along $Y_i$ is 
\[ {\rm Res}_{Y_i} \, \omega \, = \, \eta_{|Y_i}.\]
To compute the decomposition \eqref{eq:decompomega}, one must express $\omega$ in local coordinates. For instance, if $\mathbb{C}^d \subseteq X$ is an affine open subset, one can use an expression like \eqref{eq:g/f} and observe that 
\[ \omega \, = \,  \frac{g}{f_1 \cdots f_r} \, {\rm d} x_1 \wedge \cdots \wedge {\rm d} x_{d-1} \wedge {\rm d} x_d \, = \, \frac{g}{f_1 \cdots \hat{f}_i \cdots f_{r} \frac{\partial f_i}{\partial x_d}} \, {\rm d} x_1 \wedge \cdots \wedge {\rm d}x_{d-1} \wedge  \frac{{\rm d} f_i}{f_i}. \]
%More generally, for a small neighborhood $U$ of a smooth point on $X$ one defines a chart $\varphi: U \rightarrow V \subset \mathbb{C}^d$ via the implicit function theorem. On $V$, the form $\omega$ takes the form $h(x) \, {\rm d} x$ for some meromorphic function $h(x)$.

\begin{example} \label{ex:residues}
    Here are some simple examples of residue computations: 
    \[ {\rm Res}_{x=0} \Big (\frac{1}{xy} \, {\rm d}x \wedge {\rm d}y \Big )\, = \,  {\rm Res}_{x=0} \Big ( \frac{1}{x(y-x)} \, {\rm d}x \wedge {\rm d}y \Big ) \, = \, -\frac{{\rm d}y}{y}, \quad {\rm Res}_{y=0} \Big (\frac{1}{xy} \, {\rm d}x \wedge {\rm d}y \Big) \, = \, \frac{{\rm d}x}{x}. \qedhere\]
\end{example}

\begin{definition} \label{def:posgeomABL}
    With the above notation and assumptions, the tuple $(X,X_{\geq 0})$ is called a \emph{positive geometry} if there exists a unique nonzero rational differential $d$-form $\omega(X_{\geq 0})$ on $X$, called \emph{canonical form},  satisfying the following recursive axioms: 
    \begin{enumerate}
        \item If $d = 0$, then $X = X_{\geq 0} = {\rm pt}$ and $\omega(X_{\geq 0}) = \pm 1$. The sign is the orientation of $X$. 
        \item If $d > 0$, then $\omega(X_{\geq 0})$ has simple poles along $Y_1, \ldots, Y_r$, with $Y_i \not \subseteq {\rm Sing}(X)$, and it has no other poles on $X \setminus {\rm Sing}(X)$. Moreover, $r >0$ and the tuples $(Y_i, (Y_i)_{\geq 0})$ are $(d-1)$-dimensional positive geometries with canonical forms $\omega((Y_i)_{\geq 0}) = {\rm Res}_{Y_i} \, \omega(X_{\geq 0})$. 
    \end{enumerate}
\end{definition}

The requirement that $Y_i \not \subseteq {\rm Sing}(X)$ in each step of the recursion of Definition \ref{def:posgeomABL} is quite restrictive, but necessary for our definition of the residue to make sense. This restriction can be dropped, but it makes the definition more technical. We shall explain this by examples (Examples \ref{ex:cusp} and \ref{ex:nodalexample}). First, we illustrate the definition as stated. 

\begin{example} \label{ex:linesegments}
    Let $X = \mathbb{P}^1$ and let $X_{\geq 0} = [a,b]$ with $a < b$ be a real interval on $X(\mathbb{R}) = \mathbb{RP}^1$. We assume that $X_{\geq 0} \subset \mathbb{C} \subset \mathbb{P}^1$ and $\mathbb{C}$ has coordinate $x$. The algebraic boundary of $X_{\geq 0}$ consists of two components $Y_1 = \{a \}$, $Y_2 = \{b\}$. Consider the rational one-form
    \[ \omega \, = \, \frac{b-a}{(x-a)(b-x)} \, {\rm d} x \, = \, {\rm dlog}(x-a) - {\rm dlog}(x-b).\]
    This form only has poles on $Y$, and these poles are simple ($\omega$ is a \emph{logarithmic} one-form on $X \setminus Y$). Moreover, any form with this property is a nonzero multiple of $\omega$. We check that ${\rm Res}_{a} \, \omega = 1$ and ${\rm Res}_b \,  \omega = -1$. We conclude that $(X,X_{\geq 0})$ is a positive geometry with $\omega(X_{\geq 0}) = \omega$. The orientation of ${\rm int}(X_{\geq 0}) = (a,b)$ is by ${\rm d} x$, i.e., from left to right.
\end{example}
\begin{example} \label{ex:posorth}
    Let $x$ be a coordinate on $\mathbb{C} \subset X = \mathbb{P}^1$ and let $X_{\geq 0}$ be the closure of $\mathbb{R}_+$ in $X$. One checks that $(X,X_{\geq 0})$ is a positive geometry with canonical form $\omega = \frac{{\rm d}x}{x}$. Indeed, $\omega$ has simple poles at $\{0, \infty\}$ and the residues are $\pm 1$ (this is basically Example \ref{ex:linesegments}). This positive geometry is often denoted by $\mathbb{P}^1_{\geq 0}$; the ``nonnegative points'' of $\mathbb{P}^1$.
    
    We increase the dimension by one, we set $X = \mathbb{P}^2$ and $X_{\geq 0}$ is the closure of the positive points $\mathbb{R}_+^2$. The algebraic boundary $Y$ consists of three divisors $Y_0$, $Y_1$ and $Y_2$, given by the vanishing of the three homogeneous coordinates $x_0, x_1, x_2$ respectively. Using coordinates $x = x_1/x_0,y = x_2/x_0$ on $\mathbb{C}^2 \subset \mathbb{P}^2$, we consider the rational two-form
    \[ \omega \, = \, \frac{1}{xy} \, {\rm d}x \wedge {\rm d}y. \]
    One checks that this has simple poles along each of the boundary components.
    Its residues along $Y_1$ and $Y_2$ were computed in Example \ref{ex:residues}. These are indeed the canonical forms of $\mathbb{P}^1_{\geq 0}$ with the appropriate orientation, where the positive quadrant is oriented counterclockwise. More generally, $(\mathbb{P}^d, \mathbb{P}^d_{\geq 0})$ is a positive geometry with $\omega(\mathbb{P}^d_{\geq 0}) =(x_1 \cdots x_d)^{-1} {\rm d}x_1 \wedge \cdots \wedge x_d$.
\end{example}

\begin{example} \label{ex:triangle}
    The non-negative part of a projective space is combinatorially a simplex. That is, $\mathbb{P}^d_{\geq 0}$ is a $d$-dimensional convex polytope with $d+1$ vertices in the real part of an affine chart containing it. For instance, in the chart $x_0 + x_1 + x_2 \neq 0$, the nonnegative projective plane $\mathbb{P}^2_{\geq 0}$ is the triangle bounded by $x \geq 0, y \geq 0, 1-x-y \geq 0$, where $x = x_1/(x_0+x_1+x_2)$ and $y = x_2/(x_0+x_1+x_2)$. In these coordinates, the canonical form is 
    \[ \omega \, = \, \frac{1}{xy(1-x-y)} \, {\rm d}x \wedge {\rm d} y.\]
    The three edges of the triangles are positive geometries in $\mathbb{P}^1$ whose canonical forms are the residues of $\omega$. We encourage the reader to check this explicitly. 
\end{example}

\begin{example}
    Here is a non-example: the real projective line $\mathbb{RP}^1$ is \emph{not} a positive geometry: its boundary is empty. A disk in $\mathbb{R}^2 \subset \mathbb{P}^2$ is not a positive geometry: its boundary is~$\mathbb{RP}^1$.
\end{example}

\begin{example} \label{ex:circlesegm}
    The circle $X$ in Figure \ref{fig:firstexamples} is parametrized by 
    \[ x(t) \, = \, \frac{2t}{t^2+1}, \quad y(t) \, = \, \frac{1-t^2}{t^2+1}.\]
    This parametrization identifies the curve segment $X_{\geq 0}$ with the line segment $[0,1]$, oriented from right to left. We claim that $(X, X_{\geq 0})$ is a positive geometry whose canonical form can be expressed as the restriction of the one-form 
    \[ \frac{1}{(1-x-y)} \, \frac{{\rm d}x}{y} \]
    to $X$. To check this, we compute the pullback of this form along $\varphi(t) = (x(t),y(t))$: 
    \[ \varphi^*\Big (\frac{1}{(1-x-y)} \, \frac{{\rm d}x}{y} \Big)_{|X} \, = \, \frac{t^2+1}{2(t^2-t)} \cdot \frac{t^2+1}{1-t^2} \cdot \frac{2(t^2+1)-(2t)^2}{(t^2+1)^2} \, {\rm d} t  \, = \, \frac{1}{t(t-1)} \, {\rm d}t.\]
    This is the correct form by Example \ref{ex:linesegments}.
\end{example}

\begin{example} \label{ex:pizza}
    We now replace the edge of the triangle in Example \ref{ex:triangle} contained in the line $1-x-y =0$ by a segment of a conic, as in the right part of Figure \ref{fig:firstexamples}. For simplicity, let us use the conic $Y_3$ locally given by $x^2 + y^2 = 1$ (this is the circle in the left part of Figure \ref{fig:firstexamples}). We claim that the resulting shape $X_{\geq 0}$ gives a positive geometry with canonical form 
    \begin{equation} \label{eq:canformpizza} \omega \, = \, \frac{1+x+y}{xy(1-x^2-y^2)} \, {\rm d}x \wedge {\rm d} y. \end{equation}
    This form clearly has the correct poles. The residues along $Y_1$ and $Y_2$ are easy: 
    \[ {\rm Res}_{x=0} \, \omega \, =\, \frac{-1}{y(1-y)} \,  {\rm d}y ,\quad {\rm Res}_{y=0} \, \omega \, = \, \frac{1}{x(1-x)} \, {\rm d}x.\]
    These are indeed the canonical forms of the boundary line segments, with the correct induced orientation. For the residue along $Y_3$ we compute that 
    \[ {\rm Res}_{Y_3} \, \omega \, = \, {\rm Res}_{Y_3} \Big (\frac{1+x+y}{xy(-2y)} \, {\rm d} x \wedge \frac{{\rm d}(1-x^2-y^2)}{1-x^2-y^2} \Big )\, = \, \Big ( - \frac{1 + x + y}{2xy^2} \, {\rm d} x \Big )_{|Y_3}\]
    To match this with what we found in Example \ref{ex:circlesegm}, notice that 
    \[\Big ( - \frac{1 + x + y}{2xy^2} \, {\rm d} x \Big )_{|Y_3} \, = \, \Big ( - \frac{(1 + x + y)(1-x-y)}{2xy^2(1-x-y)} \, {\rm d} x \Big )_{|Y_3} \, = \, \Big(\frac{1}{y(1-x-y)} \, {\rm d} x \Big )_{|Y_3}.   \qedhere \]
\end{example}

\begin{example} \label{ex:cusp}
This example, taken from \cite[Figure 4]{kohn2025adjoints}, illustrates one way to circumvent the assumption $Y_i \not \subseteq {\rm Sing}(X)$. The variety $X$ is $\mathbb{P}^2$, and $X_{\geq 0}$ is the shaded green area in Figure \ref{fig:cuspexample}. The boundary divisor $Y$ consists of two components: the parabola given in local coordinates by $y = x^2$, and the cuspidal cubic given by $x^2=y^3$. Note that $X$ is smooth and the assumption $Y_i \not \subseteq {\rm Sing}(X)$ is satisfied, but we will see that a problem occurs after taking the residue along $Y_2$. 
    \begin{figure}
        \centering
        \includegraphics[width=0.5\linewidth]{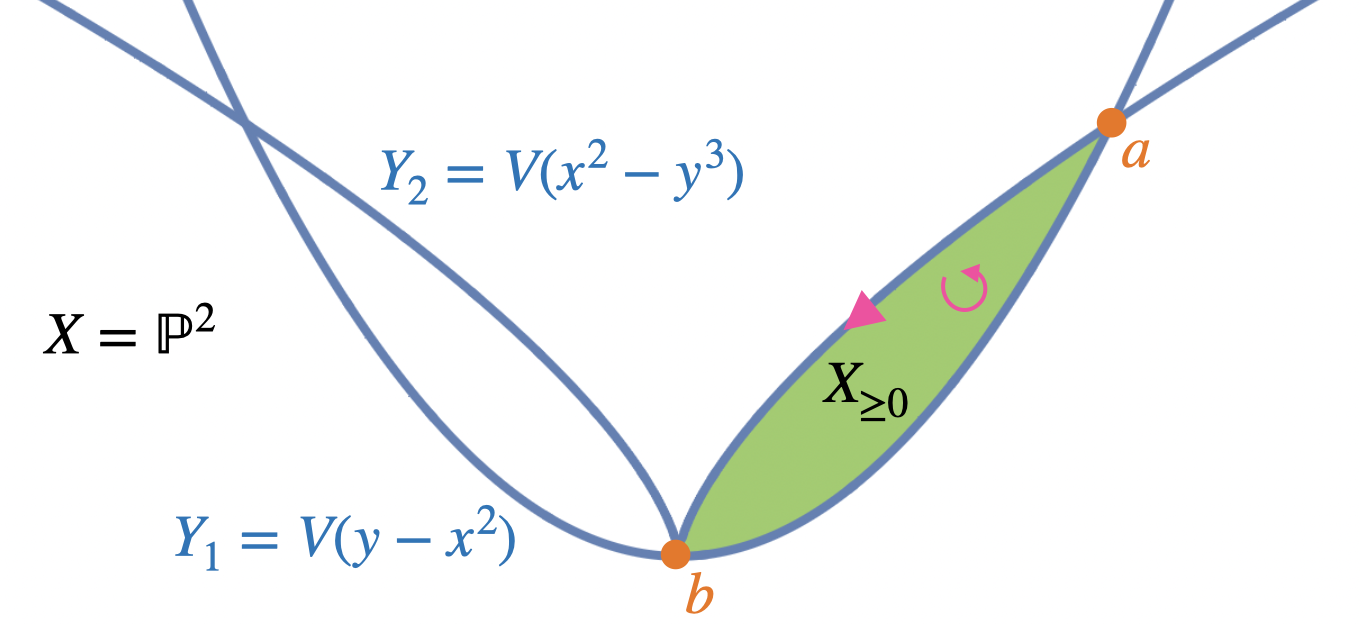}
        \caption{A positive geometry bounded by a parabola and a cuspidal cubic.}
        \label{fig:cuspexample}
    \end{figure}
We shall try to match Definition \ref{def:posgeomABL} with 
\[ \omega(X_{\geq 0}) \, = \, \frac{x^2+y^2+xy+x}{(y-x^2)(x^2-y^3)}\,  {\rm d} x \wedge {\rm d} y.\]
This form clearly has the correct poles, so we must only check the recursive condition on the residues. To compute ${\rm Res}_{Y_2} \, \omega(X_{\geq 0})$, we write our form as follows: 
\[ \omega(X_{\geq 0}) \, = \, \Big ( \frac{x^2+y^2+xy+x}{(y-x^2)(-3y^2)} \, {\rm d} x  \Big ) \wedge \frac{{\rm d}(x^2-y^3)}{x^2-y^3}.\]
The residue ${\rm Res}_{Y_2} \, \omega(X_{\geq 0})$ is the restriction of the following one-form to the curve $Y_2$: 
\[ \eta \, = \, \frac{x^2+y^2+xy+x}{(y-x^2)(-3y^2)} \, {\rm d} x.\]
In order to satisfy the definition, this restriction $\eta_{|Y_2}$ should be the canonical form of the curve segment connecting the points $a$ and $b$ along the cuspidal cubic $Y_2$. However, one of the components of the boundary of this segment, namely, the point $b$, is contained in the singular locus of $Y_2$. To proceed, we must define what it means to ``take the residue of $\eta_{|Y_2}$ at $b$''. 

One way to do this is by resolving the cusp singularity. Consider the map $\pi: \mathbb{C} \rightarrow Y_2$ given by $t \mapsto (t^3,t^2)$. The closure of the pre-image of ${\rm int}((Y_2)_{\geq 0})$ is the line segment $[0,1]$, oriented from right to left. We have $0 = \pi^{-1} (b)$ and $1 = \pi^{-1}(a)$. We require that $[0,1]$ is a positive geometry whose canonical form is the pullback $\pi^*(\eta_{|Y_2})$. The form $\pi^*(\eta_{|Y_2})$ is a rational form on the smooth curve $\mathbb{C}$, computed by substituting $x = t^3, y = t^2$ into $\eta$: 
\[ \pi^*(\eta_{|Y_2}) \, = \, \frac{t^6 + t^4 + t^5 + t^3}{(t^2-t^6)(-3t^4)} \, 3 t^2 \, {\rm d}t \, = \, \frac{1}{t(t-1)} \, {\rm d} t.\]
This is indeed the canonical form of $[0,1]$ with orientation $-{\rm d}t$, see Example \ref{ex:linesegments}.

In general, if a component of $Y$ is contained in the singular locus of $X$, one can consider the normalization map $\pi: X' \rightarrow X$ and replace $X_{\geq 0}$ by the closure $X'_{\geq 0}$ of $\pi^{-1}({\rm int}(X_{\geq 0}))$ in $X'(\mathbb{R})$. The singular locus of the normal variety $X'$ has codimension $\geq 2$, so no component of the algebraic boundary of $X'_{\geq 0}$ is contained in it. One can adapt Definition \ref{def:posgeomABL} by calling $(X,X_{\geq 0})$ a positive geometry with canonical form $\omega$ if the form $\pi^*\omega$ satisfies the recursive residue condition for $(X',X'_{\geq 0})$. More resolutions of singularities may be necessary further down in the recursion. Spelling this out more formally is beyond the scope of these lectures.
\end{example}

\begin{example} \label{ex:nodalexample}
    This example hints at another way of making Definition \ref{def:posgeomABL} more inclusive. Namely, by dealing in a more satisfactory way with singular points of $Y_i$ on the boundary $\partial X_{\geq 0}$. Consider the nodal cubic $Y \subset X = \mathbb{P}^2$ locally defined by $y^2 = x^2(x+1)$. Let $X_{\geq 0}$ be the semi-algebraic set given by $\{ (x,y) \in \mathbb{R}^2 \, : \, y^2-x^2(x+1) \leq 0 \text{ and } x \leq 0 \}$. This is the bounded ``drop-shaped'' region of the complement of $Y$ in $\mathbb{R}^2$. Consider the differential form 
    \[ \omega \, = \, \frac{2}{y^2-x^2(x+1)} \, {\rm d} x \wedge {\rm d} y.\]
    The boundary $Y$ consists of a single component, and $\omega$ has a simple pole along $Y$. The residue along $Y$ is the restriction of $\eta = \frac{{\rm d}x}{y}$ to $Y$. This has no pole on $Y \setminus {\rm Sing}(Y)$, but $Y_{\geq 0}$ has no boundary $(r = 0)$. Hence, $(X,X_{\geq 0})$ is not a positive geometry according to Definition \ref{def:posgeomABL}. 
    
    To remedy this, we blow up $\mathbb{P}^2$ at $(0,0)$ and write the blow-up map as $\pi: X' \rightarrow X$. We let $X'_{\geq 0}$ be the closure of $\pi^{-1}({\rm int}(X_{\geq 0}))$ in $X'(\mathbb{R})$, see Figure \ref{fig:nodeexample}. The algebraic boundary of $X'_{\geq 0}$ consists of the exceptional divisor $Y_0$ and the strict transform $Y'$ of $Y$. We check that $(X', X'_{\geq 0})$ is a positive geometry in the sense of Definition \ref{def:posgeomABL}, with canonical form $\omega' = \pi^* \omega$.
    \begin{figure}
        \centering
        \includegraphics[width=0.35\linewidth]{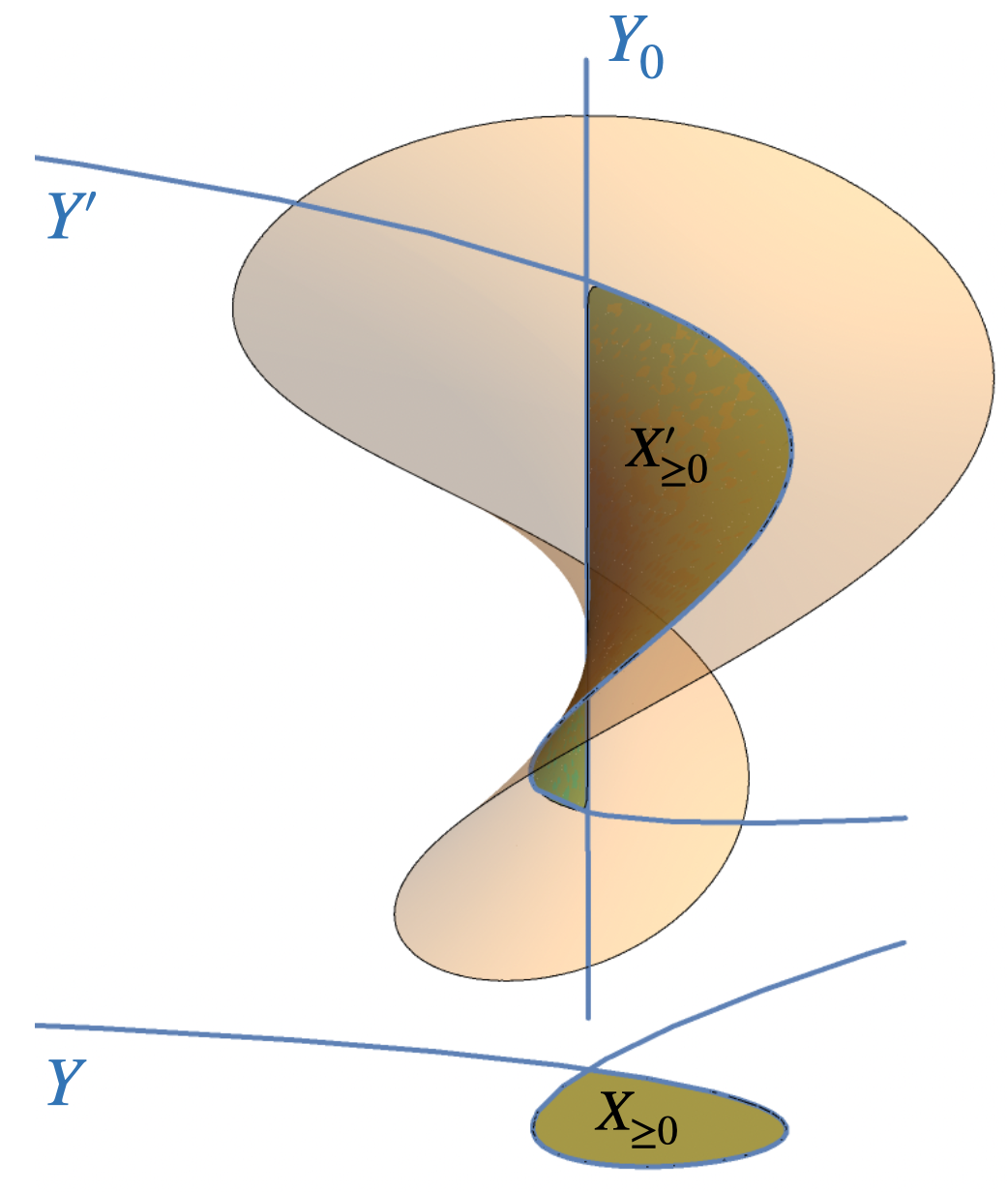}
        \caption{A positive geometry after blowing up a node.}
        \label{fig:nodeexample}
    \end{figure}
    The surface $X'$ is locally given by the equation $zx-y = 0$. Notice that $Y_0$ is locally given by $x=0$ and we can use $z$ as a local coordinate on $Y_0$ to compute ${\rm Res}_{Y_0} \omega'$. We calculate that
    \[ \omega' \, = \, \frac{2}{y^2-x^2(x+1)} \, {\rm d} x \wedge {\rm d} y \, = \, \frac{2\, x}{y^2-x^2(x+1)} \, {\rm d} x \wedge {\rm d} z \, = \, \frac{2}{z^2-x-1} \, \frac{{\rm d} x}{x} \wedge {\rm d} z.\]
    Hence ${\rm Res}_{Y_0} \omega' = -2(z^2-1)^{-1} \, {\rm d}z$. This is the canonical form of the line segment $(Y_0)_{\geq 0}$ by Example \ref{ex:linesegments}. For the residue along $Y'$, we use the parametrization $t \mapsto (t^2-1,t^3-t,t)$ of $Y'$. This identifies $Y'_{\geq 0}$ with the line segment $[-1,1]$, oriented from right to left. We have 
    \[ \Big ( \frac{1}{y} \, {\rm d} x \Big )_{|Y'} \, = \, \frac{1}{(t^3-t)} \, {\rm d}(t^2-1) \, = \, \frac{2}{t^2-1} \, {\rm d}t.\]
    Again, this is the correct form by Example \ref{ex:linesegments}.
\end{example}

The positive geometries encountered in later lectures align with Definition \ref{def:posgeomABL}. For completeness, we shall now give another candidate definition of a positive geometry which is inspired by (but not explicitly stated in) the paper \cite{brown2025positive}. The data going into this definition are a $d$-dimensional irreducible complex projective variety $X$ and a closed subvariety $Y \subsetneq X$ such that $X \setminus Y$ is smooth. A \emph{relative $d$-cycle} of the pair $(X,Y)$ is a $d$-dimensional chain in $X$ with boundary contained in $Y$. The homology classes of such chains form the relative homology group $H_d(X,Y)$. The pair $(X,Y)$ is said to have \emph{genus zero} if the Hodge numbers $h^{-p,0}$ of $H_d(X,Y)$ for $p >0$ are all equal to zero. For such genus zero pairs, Brown and Dupont observe \cite[Definition 2.6]{brown2025positive} that there is a canonical~map 
\begin{equation} \label{eq:browndupont}
    \omega: H_d(X,Y) \longrightarrow \Omega_{\rm log}^d(X\setminus Y)
\end{equation}
which associates a logarithmic differential form $\omega(\sigma)$ to a relative homology class $\sigma$. If $Y \subset X$ is a divisor, i.e., $Y$ is pure of dimension $d - 1$, then $\Omega_{\rm log}^d(X\setminus Y)$ is contained in the rational forms on $X$ with at most simple poles along the irreducible components of $Y$.

\begin{definition} \label{def:posgeomBD}
    A \emph{positive geometry} is a relative homology class $\sigma \in H_d(X,Y)$ of a genus zero pair $(X,Y)$. The \emph{canonical form} of $\sigma$ is $\omega(\sigma)$, where $\omega$ is the map from \eqref{eq:browndupont}.
\end{definition}

We point out that \cite[Section 3.2]{brown2025positive} provides many tools for determining the genus of a pair $(X,Y)$. In particular, all pairs $(X,Y)$ in the above examples have genus zero. We encourage the reader to show this using Sections 3.3.2 and 3.3.3 in \cite{brown2025positive}.

It is clear that Definitions \ref{def:posgeomABL} and \ref{def:posgeomBD} are not equivalent. A pair $(X,X_{\geq 0})$ for which $\sigma = [X_{\geq 0}] \in H_d(X,Y)$ is a positive geometry in the sense of Definition \ref{def:posgeomBD} might not be a positive geometry in the sense of Definition \ref{def:posgeomABL}. For instance, $X_{\geq 0}$ might not be contained in the real points of $X$. Moreover, Definition \ref{def:posgeomBD} does not require $\omega(\sigma)$ to be nonzero. A more delicate example arises from the fact that a ``boundary pair'' of a genus zero pair $(X,Y)$ need not have genus zero, see \cite[Section 5.9]{brown2025positive}. Hence, positive geometries in the sense of Definition \ref{def:posgeomBD} may have boundaries which are not positive geometries. However, for boundary pairs which are genus zero, the recursive property from Definition \ref{def:posgeomABL} holds: the residues along the boundary components are the canonical forms of the boundaries. Moreover, if this holds for all boundaries, then the residues determine $\omega(\sigma)$ uniquely, and $\omega(\sigma)$ must equal the canonical form in the sense of Definition \ref{def:posgeomABL}. For precise statements, see \cite[Propositions 2.14 and 2.15]{brown2025positive}. This makes it quite hard, if not impossible, to find an example of a positive geometry in the sense of Definition \ref{def:posgeomABL} which does not fit Definition \ref{def:posgeomBD}. It would be interesting to find such an example or to prove that this is impossible.

In summary, to my knowledge, Definition \ref{def:posgeomABL} is more restrictive than Definition \ref{def:posgeomBD} and it is more technical to state, especially if one wants to allow the positive geometries appearing in Examples \ref{ex:cusp} and \ref{ex:nodalexample}. On the other hand, Definition \ref{def:posgeomBD} does not guarantee the full recursive structure imposed by Definition \ref{def:posgeomABL}. This is perhaps undesirable in the physics application, which is a potential criticism for my interpretation in Definition \ref{def:posgeomBD}. Maybe not \emph{all} genus zero pairs should be called positive geometries. Another such criticism is that Definition \ref{def:posgeomBD} does not talk about the real points of $X$ and $Y$ (it is hardly \emph{positive}). This discussion confirms that positive geometry is still in its infancy; we have yet to determine the correct definition(s). Nevertheless, there is a consensus that certain examples should certainly fall under any reasonable definition of ``positive geometry''. Among those are polyhedra in projective space $X = \mathbb{P}^d$. Their canonical forms are the topic of the next section.

\begin{remark}
    The papers \cite{arkani2017positive,lam2024invitation} provide a zoo of examples of positive geometries. In these notes, we limit ourselves to the well understood cases of convex polytopes (Section \ref{sec:2}) and planar positive geometries (Section \ref{sec:3}). Many positive geometries of interest in physics are semi-algebraic subsets of the real points of a Grassmannian. For instance, the tree-level amplituhedron ${\cal A}_{k,n,m}$ lives in $X = {\rm Gr}(k,k+m)$. See \cite{de2024amplituhedron} for a gentle introduction. While the amplituhedron is one of the main motivations for the development of positive geometry, to my knowledge it is only known that ${\cal A}_{k,n,m}$ is a positive geometry in the sense of Definition \ref{def:posgeomABL} for $k = 1$, in which case it is a cyclic polytope, and for $k = m = 2$ \cite{ranestad2024adjoints}. It is conjectured to be a positive geometry for arbitrary $k,n,m$. For more on positive geometries in $X = {\rm Gr}(2,4)$, see \cite{pavlov2025positive}. For positive geometries in the wonderful compactification $X$ of a hyperplane arrangement complement, see \cite{brauner2024wondertopes}. Lam discusses the positive geometry of the moduli space $X = \overline{{\cal M}}_{0,n}$ in \cite{lam2024moduli}. Finally, in \cite{Early2023Positive}, $X$ is a moduli space of del Pezzo surfaces. 
\end{remark}

\section{Positive geometry of polytopes} \label{sec:2}

A (convex) \emph{polytope} in $\mathbb{R}^d$ is the convex hull of a finite set of points. That is, it is the smallest convex set containing these points. Alternatively, a polytope can be described as the set of solutions to a finite set of affine-linear inequalities, and a set of such inequalities defines a polytope as long as the solution set is bounded. A list of $n$ affine-linear inequalities in $d$ variables is represented by a matrix $U \in \mathbb{R}^{n \times d}$ of rank $d$, and a vector $z \in \mathbb{R}^n$. Each row $u_i \in (\mathbb{R}^d)^\vee$ of $U$ and entry $z_i \in \mathbb{R}$ of $z$ define one of the inequalities: 
\begin{equation} \label{eq:facetrep} P \, = \, \{ y \in \mathbb{R}^d \, : \, U \cdot y + z \geq 0 \}, \end{equation}
where $U \cdot y + z \geq 0$ is short for $u_i \cdot y + z_i \geq 0, \, i = 1, \ldots, n$. The dimension of $P$ is the dimension of the smallest affine space containing $P$. For a vector $u \in (\mathbb{R}^d)^\vee$, we define the set $P_u = \{ y \in P\, : \, u \cdot y = {\rm min}_{y' \in P} \, u \cdot y' \}$, which is itself a convex polytope. All polytopes arising in this way are called \emph{faces} of $P$. Faces of dimension $\dim P - 1$ are called \emph{facets}. 
\begin{example} \label{ex:polytopes}
    Polytopes in $\mathbb{R}^1$ are line segments. Polytopes in $\mathbb{R}^2$ are also called \emph{polygons}. An example is shown in the left part of Figure \ref{fig:polytopes}. This pentagon $P$ is the convex hull of its \emph{vertices} (i.e., zero-dimensional faces) $\{(-1,-1),(0,-1),(1,0),(1,1),(-1,1)\}$. We have 
    \[ P \, = \, \{ y \in \mathbb{R}^2 \, : \, U \cdot y + z \geq 0 \},\]
    where $U \in \mathbb{R}^{5 \times 2}$ and $z \in \mathbb{R}^5$ are given by 
    \[ U \, = \, \begin{pmatrix}
        1 & 0 & -1 & -1 & 0 \\ 0 & 1 & 1 & 0 & -1
    \end{pmatrix}^t, \quad z \, = \, \begin{pmatrix} 1 & 1 & 1 & 1 & 1 \end{pmatrix}^t.\]
    The inequalities read $y_1 + 1 \geq 0, y_2 + 1 \geq 0, -y_1 + y_2 +1 \geq 0, -y_1 + 1 \geq 0, -y_2 + 1 \geq 0$. The rows of $U$ are the \emph{inward pointing facet normals} of $P$. 
\begin{figure}
    \centering
    \includegraphics[width=0.5\linewidth]{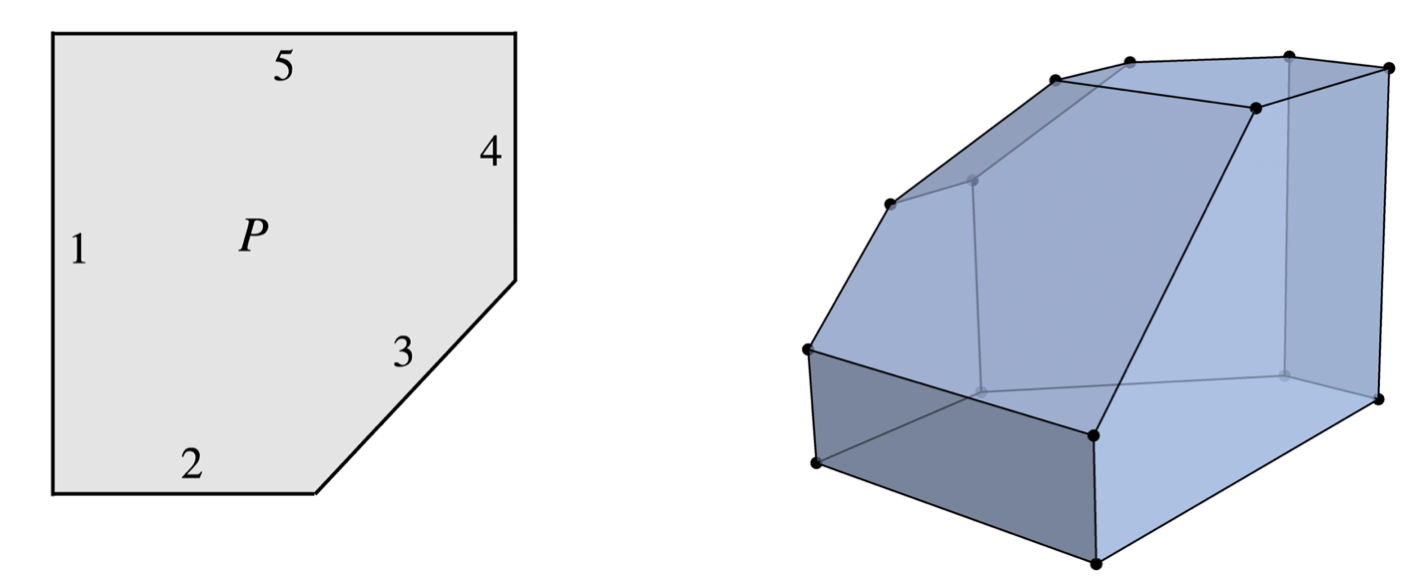}
    \caption{Two polytopes of dimension $d = 2$ (left) and $d = 3$ (right). Taken from \cite{telen2025toric}.}
    \label{fig:polytopes}
\end{figure}

The right part of Figure \ref{fig:polytopes} shows a three-dimensional polytope given by $U \cdot y + z \geq 0$ with 
\[ U \, = \, \begin{pmatrix}
    -1 & 0 & 0 & 1 & 1 & 1 & 0 & 0 & 0\\
    0 & -1 & 0 & -1 & 0 & 0 & 1 & 1 & 0 \\ 
    0 & 0 & -1 & 0 & -1 & 0 & -1 & 0 & 1
\end{pmatrix}^t\]
and $z = (3,4,3,2,2,0,1,0,0)$. You will compute its vertices as an exercise (Exercise \ref{ex:vertassoc}). 
\end{example}
We will assume that $P \subset \mathbb{R}^d$ has dimension $d$, and that the representation \eqref{eq:facetrep} of $P$ is \emph{minimal}, which means that the rows of $U$ are the inward pointing normal vectors to the facets of $P$, as in Example \ref{ex:polytopes}. The set of facets of $P$ is denoted by ${\cal F}(P)$. We write $u_F$ for the row of $U$ corresponding to the facet $F = P_{u_F} \in {\cal F}(P)$, and $z_F$ for the corresponding entry of $z$. For each facet $F \in {\cal F}(P)$, let $H_F \subset \mathbb{P}^d$ be the complex hyperplane obtained as the Zariski closure of $\{y \in \mathbb{R}^d \, : \, u_F \cdot y + z_F = 0\}$. The union of these hyperplanes is the \emph{facet hyperplane arrangement} $Y_P = \bigcup_{F \in {\cal F}(P)} H_F$. The set of vertices of $P$ is denoted by ${\cal V}(P)$. For each vertex $v \in {\cal V}(P)$, we write $U_v$ for the submatrix of $U$ consisting of rows $u_F$ for which $v \in F$. The polytope $P$ is called \emph{simple} if each vertex $v \in {\cal V}(P)$ is contained in precisely $d$ facets. If $P$ is simple, then each matrix $U_v, v \in {\cal V}(P)$ is square and invertible. We will prove the following:

\begin{theorem} \label{thm:canformpolytope}
    Let $P \subset \mathbb{R}^d \subset \mathbb{RP}^d$ be a $d$-dimensional simple polytope. We have that $(\mathbb{P}^d, P)$ is a positive geometry in the sense of Definition \ref{def:posgeomABL}. Its canonical form is 
    \[ \omega(P) \, = \, \sum_{v \in {\cal V}(P)} \frac{ |\det U_v | }{\prod_{\substack{F \in {\cal F}(P) \\ v \in F }} (u_F \cdot y + z_F)} \, {\rm d }y.\]
    Here the product is over facets containing $v$. Moreover, the class $\sigma_P = [P] \in H_d(\mathbb{P}^d, Y_P)$ is a positive geometry in the sense of Definition \ref{def:posgeomBD}. The image of $\sigma_P$ under \eqref{eq:browndupont} equals $\omega(P)$.
\end{theorem}

\begin{example} \label{ex:pentagon1}
    The reader should match the formula of Theorem \ref{thm:canformpolytope} with the canonical form of the triangle in Example \ref{ex:triangle}. The canonical form of the pentagon $P$ in Example \ref{ex:polytopes} is 
    \begin{align*} \omega(P) \, = \, & \, \frac{{\rm d}y_1 \wedge {\rm d}y_2}{(y_1+1)(y_2+1)} \, + \,  \frac{{\rm d}y_1 \wedge {\rm d}y_2}{(y_2+1)(-y_1+y_2+1)} \, + \,  \frac{{\rm d}y_1 \wedge {\rm d}y_2}{(-y_1+y_2+1)(-y_1+1)} \\[0.75em]
    & + \,   \frac{{\rm d}y_1 \wedge {\rm d}y_2}{(-y_1+1)(-y_2+1)} \, + \,    \frac{{\rm d}y_1 \wedge {\rm d}y_2}{(y_1+1)(-y_2+1)} \\[0.75em]
    \, = \, & \, \, \frac{5 - 3\, y_1 + 3 \, y_2 - y_1y_2}{(y_1+1)(y_2+1)(-y_1+y_2+1)(-y_1+1)(-y_2+1)} \, {\rm d} y_1 \wedge {\rm d} y_2. \qedhere
    \end{align*}
    %\[ \omega(P) \, = \, \frac{5 - 3\, y_1 + 3 \, y_2 - y_1y_2}{(y_1+1)(y_2+1)(-y_1+y_2+1)(-y_1+1)(-y_2+1)} \, {\rm d} y_1 \wedge {\rm d} y_2. \qedhere\]
\end{example}

\begin{proof}[Proof of Theorem \ref{thm:canformpolytope}]
    We prove the first claim by induction on $d$. The facet representation of the line segment $[a,b]$ is $\left ( \begin{smallmatrix}
        1 \\ -1
    \end{smallmatrix}\right) y + \left ( \begin{smallmatrix}
        -a \\ b
    \end{smallmatrix}\right) \geq 0$. One easily matches $\omega(P)$ with Example \ref{ex:linesegments}. Hence, the statement holds for $d = 1$.
    For $d > 1$, fix a facet $F_1 \in {\cal F}(P)$. By changing coordinates on $\mathbb{R}^d$, we may assume that $H_{F_1} = Y_1$ is given by $y_d = 0$. The residue along $Y_1$ is 
    \[{\rm Res}_{Y_1} \, \omega(P) \, = \, \left ( \sum_{v \in {\cal V}(F_1)} \frac{ |\det U_v | }{\prod_{\substack{F \in {\cal F}(P) \setminus F_1\\ v \in F }} (u_F \cdot y + z_F)} \right )_{|y_d = 0} \, {\rm d }y_1 \wedge \cdots \wedge {\rm d} y_{d-1}. \]
    The facets appearing in the denominator intersect $F_1$ in $v$. Since $P$ is simple, all such facets of $P$ intersect $F_1$ in one of its $(d-2)$-dimensional facets.
    We view $F_1$ as a polytope in $\mathbb{R}^{d-1}$ and write its minimal facet description as $U' \cdot y' + z' \geq 0$, where $y' = (y_1, \ldots, y_{d-1})$. The rows of $U'$ are $u'_F$ for $F \in {\cal F}(F_1)$. The matrix $U'$ is a submatrix of $U$ satisfying $|\det U'_v|= | \det U_v|$ for each $v \in {\cal V}(F_1) \subset {\cal V}(P)$. Here $U'_{v}$ is a $(d-1) \times (d-1)$ matrix defined for $F_1$ as $U_v$ was defined for $P$. Putting all of this together, we find the formula
    \[ {\rm Res}_{Y_1} \, \omega(P) \, = \, \sum_{v \in {\cal V}(F_1)} \frac{ |\det U'_v |}{\prod_{\substack{F \in {\cal F}(F_1) \\ v \in F }} (u'_F \cdot y' + z'_F)}  \, {\rm d }y'. \]
    By the induction hypothesis, this is the canonical form of $F_1$ as a positive geometry in~$H_{F_1}$.

    Since the pair $(\mathbb{P}^d, Y_P)$ has genus zero \cite[Proposition 3.26]{brown2025positive}, $\sigma_P = [P] \in H_d(\mathbb{P}^d, Y_P)$ is a positive geometry in the sense of Definition \ref{def:posgeomBD}. To show that the map \eqref{eq:browndupont} gives the same form, notice that $\omega(P) \in \Omega^d_{\rm log}(\mathbb{P}^d \setminus Y_P)$. This is clear for $d = 1$ (Example \ref{ex:linesegments}), and can be shown for general $d$ by inductively applying \cite[Proposition 1.15]{brown2025positive}. Then, by \cite[Propositions 2.14 and 2.15]{brown2025positive}, $\omega(P)$ equals the unique form $\omega(\sigma_P) \in \Omega^d_{\rm log}(\mathbb{P} \setminus Y_P)$ with the prescribed residues. 
\end{proof}

We shall state an analogous statement to Theorem \ref{thm:canformpolytope} for arbitrary (possibly non-simple) polytopes by introducing the \emph{dual volume function}. For a $d$-dimensional polytope $P$, we define 
\[ P^\circ \, = \, \{ u \in (\mathbb{R}^d)^\vee \, : \, u \cdot y \geq -1, \text{ for all } y \in P \}.\]
This is a convex polyhedron called the \emph{polar dual} of $P$. If $0 \in {\rm int}(P)$, then $P^\circ$ is bounded. Importantly for us, in that case, the normalized volume of $P^{\circ}$ is bounded and given by 
\[ {\rm vol}(P^\circ) \, = \, d! \cdot \int_{P^{\circ}} 1 \cdot {\rm d}u_1 \cdots {\rm d}u_d. \]
\begin{example}
The polar dual $P^\circ$ of the pentagon $P$ from Example \ref{ex:polytopes} is shown in the left part of Figure \ref{fig:dual+adjoint}. Its vertices are given by the rows of $U$. We have ${\rm vol}(P^\circ) = 5$.
    \begin{figure}
        \centering
        \includegraphics[width=0.2\linewidth]{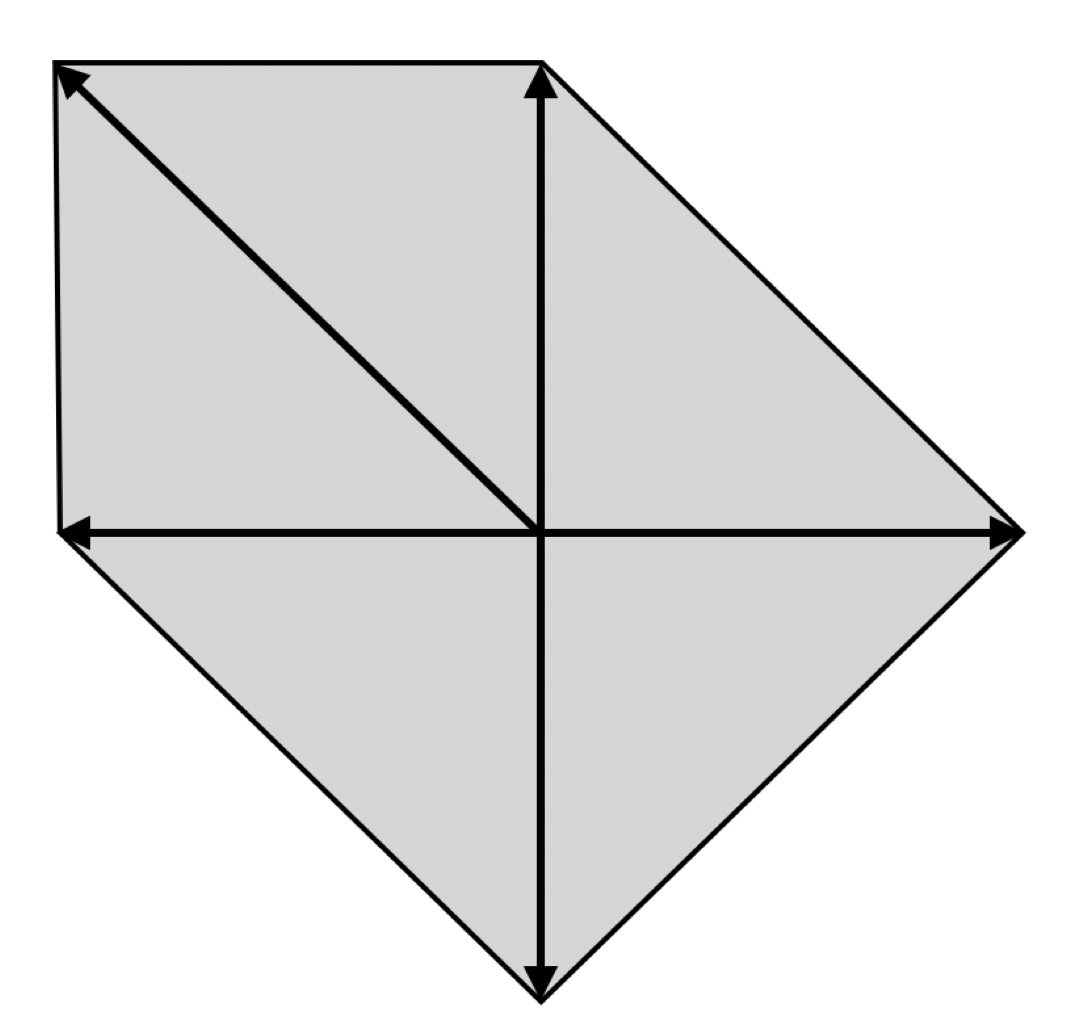} \quad \quad 
        \includegraphics[width = 0.2\linewidth]{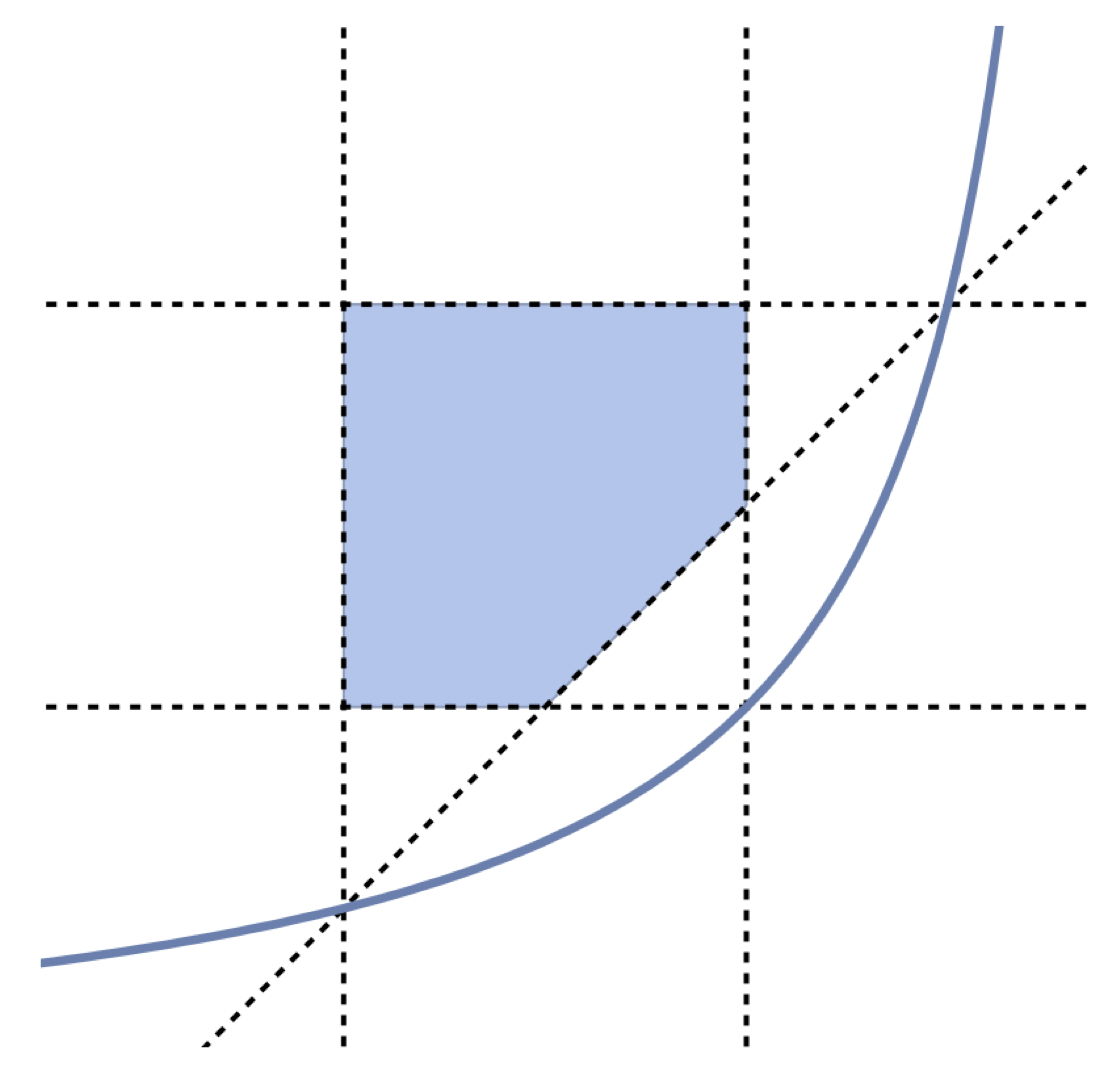}
        \caption{Left: the dual polygon of the pentagon $P$ in Example \ref{ex:polytopes}. Right: the adjoint curve of $P$. Both figures are taken from \cite{telen2025toric}.}
        \label{fig:dual+adjoint}
    \end{figure}
\end{example}
Finally, for $y \in \mathbb{R}^d$, let $P-y = \{ y' - y \, : \, y' \in P \}$ be the translate of $P$ by $y$.

\begin{theorem} \label{thm:canformpolytopesgeneral}
    Let $P \subset \mathbb{R}^d \subset \mathbb{RP}^d$ be a $d$-dimensional polytope. We have that $(\mathbb{P}^d, P)$ is a positive geometry in the sense of Definition \ref{def:posgeomABL}. Its canonical form is \[\omega(P) \, = \, f(y) \, {\rm d }y,\] where $f(y)$ is the meromorphic continuation to $\mathbb{P}^d$ of the dual volume function ${\rm int}(P) \mapsto \mathbb{R}_+$ given by $y \mapsto {\rm vol}((P-y)^\circ)$.
    Moreover, the class $\sigma_P = [P] \in H_d(\mathbb{P}^d, Y_P)$ is a positive geometry in the sense of Definition \ref{def:posgeomBD}. The image of $\sigma_P$ under \eqref{eq:browndupont} equals $\omega(P)$.
\end{theorem}

\begin{proof}
    We refer to \cite[Section 7.4]{arkani2017positive} for the statement about Definition \ref{def:posgeomABL}. The statement about Definition \ref{def:posgeomBD} can be deduced as for simple polytopes. 
\end{proof}

You will match Theorems \ref{thm:canformpolytope} and \ref{thm:canformpolytopesgeneral} in Exercise \ref{ex:dualvolume}.

\begin{remark} \label{rem:subdiv}
    A different way to express the canonical form of a polytope $P$ uses \emph{triangulations}. Let $\Delta_1, \ldots, \Delta_r \subset \mathbb{R}^d$ be $d$-dimensional simplices (polytopes with $d+1$ vertices) such that $\bigcup_i \Delta_i = P$ and ${\rm int}(\Delta_i) \cap {\rm int}(\Delta_j) = \emptyset$ for $i \neq j$. The orientation of ${\rm int}(\Delta_i)$ is induced by that of $P$. The canonical form of $(\mathbb{P}^d, \Delta_i)$ is given by Theorem \ref{thm:canformpolytope}. By the \emph{triangulation property} of canonical forms, see \cite[Section 3]{arkani2017positive} and \cite[Theorem 11]{lam2024invitation}, we have 
    \[ \omega(P) \, = \, \omega(\Delta_1) + \cdots + \omega(\Delta_r). \]
    You will verify this for a three-dimensional pyramid in Exercise \ref{ex:pyramid}.
\end{remark}

We are now ready to illustrate the connection with scattering amplitudes. 

\begin{corollary}
    Let $P \subset \mathbb{R}^d$ be $d$-dimensional and simple and let \eqref{eq:facetrep} be a minimal facet representation of $P$. Consider the rational function 
    \begin{equation} \label{eq:toricamplitude} {\rm Amp}_P(x) \, = \, \sum_{v \in {\cal V}(P)} \frac{|\det U_v|}{\prod_{\substack{F \in {\cal F}(P) \\ v \in F}} x_F } \quad \in \, \mathbb{R}(x_F \, : \, F \in {\cal F}(P)).\end{equation}
    The canonical form of $P$ is given by ${\rm Amp}_P( U \cdot y + z) \, {\rm d}  y$.
\end{corollary}
\begin{proof}
    This is an immediate consequence of Theorem \ref{thm:canformpolytope}. 
\end{proof}

The function ${\rm Amp}_P(x)$ was called the \emph{toric amplitude} of $P$ in \cite{telen2025toric}. This is motivated by the following observation. 

\begin{example}\label{ex:physics}
    The pentagon $P$ in Example \ref{ex:polytopes} is a realization of the $2$-dimensional \emph{associahedron}. In general, the $d$-dimensional associahedron is a simple polytope whose vertices are the triangulations of the $(d+3)$-gon, and whose facets represent the diagonals of the $(d+3)$-gon. A collection of $d$ facets meets in a vertex if and only if the corresponding $d$ diagonals triangulate the $n$-gon, and the triangulation corresponds to the common vertex. This is illustrated in Figure \ref{fig:assoc} for $d = 2$. 
    \begin{figure}
        \centering
        \includegraphics[width=0.37\linewidth]{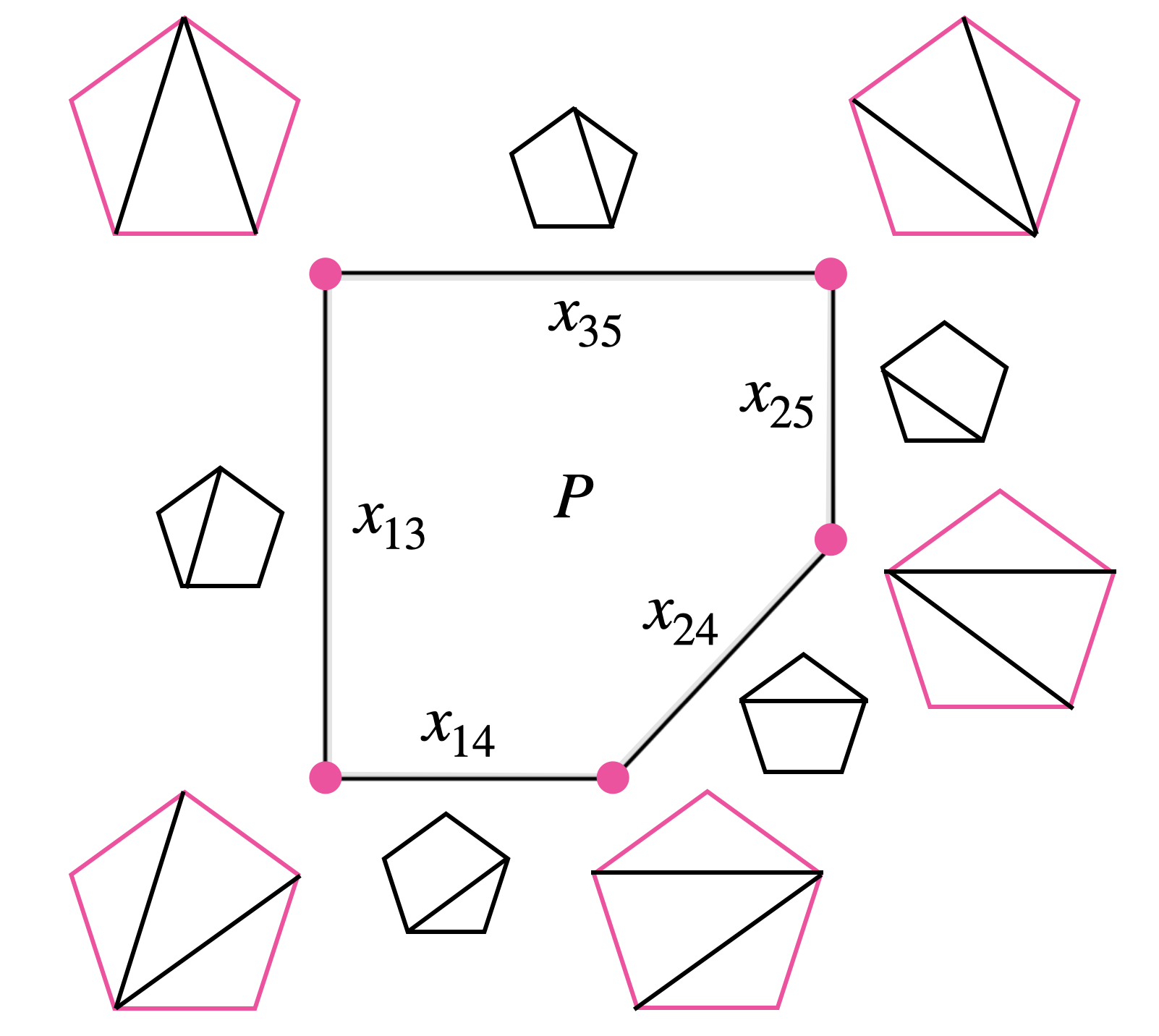}
        \caption{The $2$-dimensional associahedron.}
        \label{fig:assoc}
    \end{figure}
    Let us denote the facet variable $x_F$ from \eqref{eq:toricamplitude} by $x_{ij}$, where $F$ is the facet of the associahedron corresponding to the diagonal $ij$ of the pentagon, see the facet labeling in Figure \ref{fig:assoc}. The toric amplitude evaluates to the following function:
    \[ {\rm Amp}_P(x) \, = \, \frac{1}{x_{13}x_{14}} \, + \, \frac{1}{x_{14}x_{24}} \, + \, \frac{1}{x_{24}x_{25}} \, + \, \frac{1}{x_{25}x_{35}} \, + \, \frac{1}{x_{13}x_{35}} \,. \]
    This is the five point amplitude in biadjoint scalar $\phi^3$ theory, see \cite[Equation (3.24)]{arkani2018scattering}. To express this in terms of the momenta $p_1, \ldots, p_5$ of the five particles, one sets $x_{ij} = (p_i + \cdots + p_{j-1})^2$, where $p^2$ is the Minkowski inner product of $p$ with itself. The same construction applies for $m$ particles, where $m$ is arbitrary. The corresponding realization of the $(m-3)$-dimensional associahedron is called the \emph{ABHY associahedron}, after the authors of \cite{arkani2018scattering}.

    The right part of Figure \ref{fig:firstexamples} shows a three-dimensional ABHY associahedron. The toric amplitude is a sum over its vertices, and it equals the six point biadjoint scalar $\phi^3$ amplitude. Figure \ref{fig:assoc3} shows this function in terms of the \emph{Mandelstam variables} $s_{i, i+1} = x_{i,i+2} = (p_i + p_{i+1})^2$ and $s_{i,i+1,i+2} = x_{i,i+3} = (p_i + p_{i+1} + p_{i+2})^2$. The orange trees in the figure illustrate that the triangulations of the hexagon are dual to the planar trivalent trees with six labeled leaves. These are the \emph{Feynman diagrams}, which the amplitude is classically a sum over. 
    \begin{figure}
        \centering
        \includegraphics[width=0.9\linewidth]{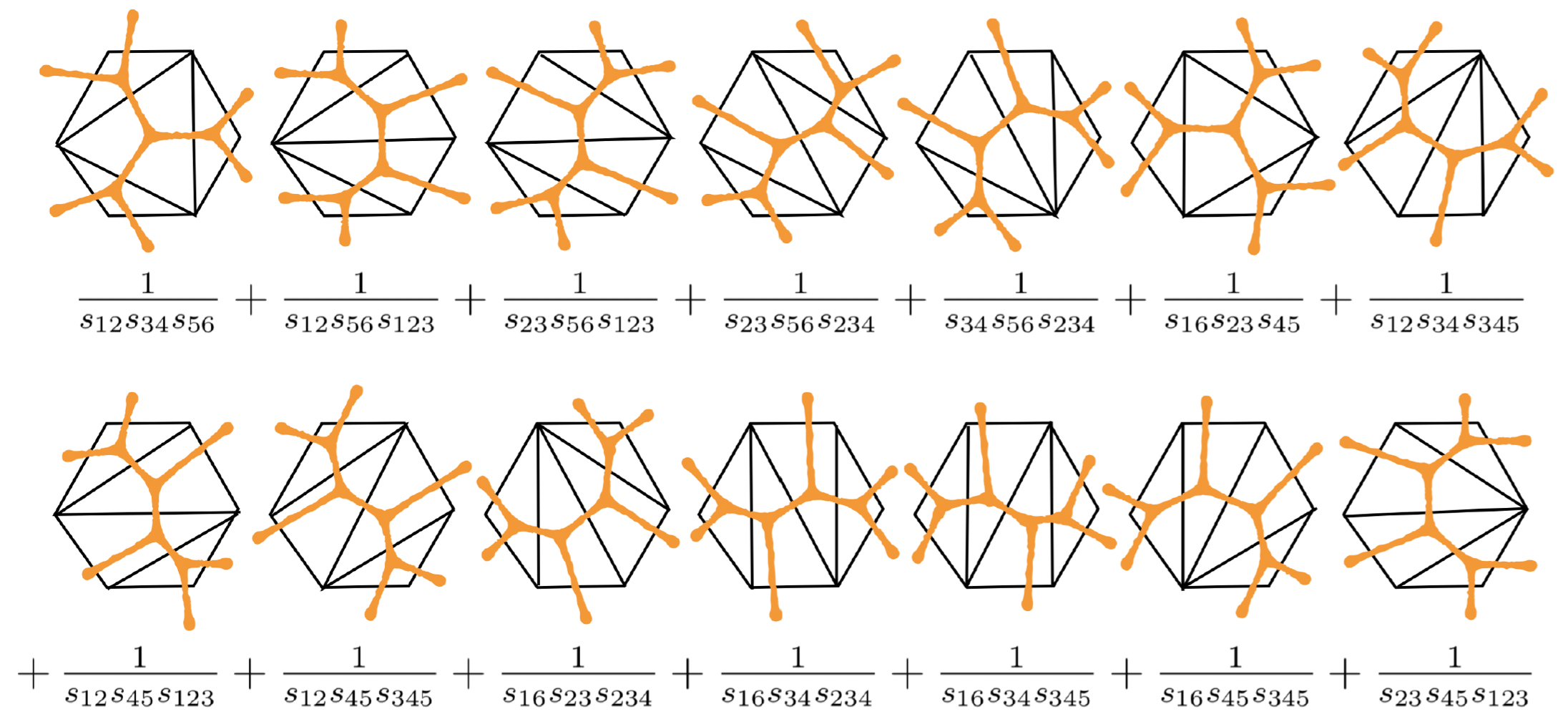}
        \caption{The amplitude as a sum over the 14 vertices of the 3D associahedron.}
        \label{fig:assoc3}
    \end{figure}
\end{example}

The \emph{poles} of the canonical form of $P$ are the facet hyperplanes $H_F, F \in {\cal F}(P)$. The \emph{zeros} form an interesting hypersurface, called the \emph{adjoint hypersurface} of $P$ \cite[Definition 2]{lam2024invitation}. To define this hypersurface, it is natural to work in homogeneous coordinates on $\mathbb{P}^d$. With a slight abuse of notation, we will switch from $y_1, \ldots, y_d$ to $y_0, y_1, \ldots, y_d$, where it is understood that $y_i$ above represents the coordinate $y_i/y_0$ in the chart $y_0 = 1$.

\begin{definition} \label{def:adjoint}
    Let $(X, X_{\geq 0})$ be a positive geometry as in Definition \ref{def:posgeomABL}. The \emph{adjoint locus} $A_{X_{\geq 0}} \subset X$ is the zero locus of the canonical form $\omega(X_{\geq 0})$. If $X = \mathbb{P}^d$, then the \emph{adjoint polynomial} of $X_{\geq 0}$ is the degree $\deg(Y)-d-1$ polynomial ${\rm adj}_{X_{\geq 0}} \in \mathbb{C}[y_0, \ldots, y_d]$ for which 
    \begin{equation} \label{eq:adjform} \omega(X_{\geq 0}) \, = \, \frac{{\rm adj}_{X_{\geq 0}}}{g(y)} \, \, \Big ( \sum_{i = 0}^d (-1)^i \, y_i \, {\rm d}y_0 \wedge \cdots \wedge \widehat{{\rm d}y_i} \wedge \cdots \wedge {\rm d}y_d \Big ),\end{equation}
    where $g(y)$ is a defining equation for the algebraic boundary $Y$ of $X_{\geq 0}$. 
\end{definition}
Let $X = \mathbb{P}^d$. By Definition \ref{def:adjoint} the adjoint locus $A_{X_{\geq 0}}$ is the zero locus of the adjoint polynomial ${\rm adj}_{X_{\geq 0}}$, which is defined up to scale. Hence, $A_{X_{\geq 0}}$ is either empty or a hypersurface. The form \eqref{eq:adjform} is a $d$-form on $\mathbb{C}^{d+1}$ which descends to a $d$-form on projective space. Moreover, all forms on projective space can be written in this way. The form between parentheses generates a cyclic $\mathbb{C}[y_0,\ldots,y_d]$-module whose coherent sheaf is the canonical sheaf of $\mathbb{P}^d$.

\begin{example}
    By Example \ref{ex:posorth}, $y^{-1} {{\rm d} y}$ is the canonical form of the closure $X_{\geq 0}$ of $\mathbb{R}_{>0}$ in $\mathbb{RP}^1 \subset X = \mathbb{P}^1$. Let $y_0, y_1$ be homogeneous coordinates on $X$ and $y = y_1/y_0$. We have 
    \[ \frac{{\rm d}y}{y} \, = \, \frac{y_0}{y_1} \, {\rm d} \Big ( \frac{y_1}{y_0} \Big ) \, = \, \frac{1}{y_0y_1} (y_0 \, {\rm d} y_1 - y_1 \, {\rm d} y_0 ).\]
    The degree of $Y$ is $2$, and the adjoint polynomial is the constant $1$. We have $A_{X_{\geq 0}} = \emptyset$.
\end{example}

\begin{proposition} \label{prop:univadjoint}
    Let $P \subset \mathbb{R}^d$ be $d$-dimensional and simple. Consider the polynomial 
    \begin{equation} \label{eq:univadjoint} {\rm Adj}_P(x) \, = \, \sum_{v \in {\cal V}(P)} |\det U_v| \cdot \prod_{\substack{F \in {\cal F}(P) \\ v \notin F}} x_F \end{equation}
    of degree $n-d$, where $n = |{\cal F}(P)|$. We have ${\rm adj}_P(y) = \frac{1}{y_0} \cdot {\rm Adj}_P(U \cdot y + z \, y_0)$.
\end{proposition}

\begin{proof}
    First, we note that $\frac{1}{y_0} \cdot {\rm Adj}_P(U \cdot y + z \, y_0)$ is indeed a polynomial: ${\rm Adj}_P(U \cdot y) = 0$ by \cite[Theorem 3.10]{telen2025toric}. Second, by Theorem \ref{thm:canformpolytope}, the divisor defined by ${\rm Adj}_P(U \cdot y + z \, y_0)$ in $\mathbb{P}^d$ restricts to that of ${\rm adj}_{P}$ on the affine chart where $y_0 \neq 0$. The degree of ${\rm adj}_{P}$ is $n-d-1$, while ${\rm Adj}_P(U \cdot y + z \, y_0)$ has degree $n-d$. This implies the proposition. 
\end{proof}

The polynomial ${\rm Adj}_P(x)$ from \eqref{eq:univadjoint} is called the \emph{universal adjoint} of $P$ in \cite{telen2025toric}. The polynomial ${\rm adj}_P$ appeared in Warren's work \cite{warren1996barycentric} on barycentric coordinates on $P$. It is sometimes called \emph{Warren's adjoint} for that reason. The universal adjoint only depends on the matrix $U$. It relates to Warren's adjoint like the toric amplitude relates to the canonical form. 

\begin{remark}
    See \cite[Theorem 3.2]{gaetz2025canonical} for a proof of the fact that, Warren's adjoint of a (not necessarily simple) polytope $P$ is the numerator of the dual volume function. 
\end{remark}

\begin{example} \label{ex:conicadjoint}
    The universal adjoint polynomial of the pentagon $P$ in Example \ref{ex:physics} is 
    \[ {\rm Adj}_P(x) \, = \, x_{24}x_{25}x_{35} + x_{13}x_{25}x_{35} + x_{13}x_{14}x_{35} + x_{13}x_{14}x_{24} + x_{14}x_{24}x_{25}. \]
    The polynomial ${\rm Adj}_P(x)$ defines a threefold in $\mathbb{P}^4$. That threefold is the zero locus of the amplitude ${\rm Amp}_P(x)$. It is known in algebraic geometry as the \emph{Segre cubic} \cite[Section 2.2]{telen2025toric}. Substituting $x_{13} = y_1 + y_0, x_{14} = y_2 + y_0, x_{24} = -y_1 + y_2 + y_0, \ldots$ in ${\rm Adj}_P$ yields 
    $y_0 (5\, y_0^2 - 3 \, y_0y_1 + 3 \, y_0y_2 - y_1y_2)$. The adjoint locus is a smooth conic in $X = \mathbb{P}^2$. It is the closure of the affine curve in $\mathbb{C}^2$ defined by the numerator $5-3y_1 + 3y_2-y_1y_2$ in Example~\ref{ex:pentagon1}. The real points of that curve are shown in the right part of Figure \ref{fig:dual+adjoint}. The curve interacts with the line arrangement $Y_P$ in an interesting way. We will now see that this is no coincidence. 
\end{example}

For a subset $S \subseteq {\cal F}(P)$, let $L_S =  \bigcap_{F \in S} H_F$. This linear space is called \emph{residual} if $L_S \cap P = \emptyset$. The union of all residual linear spaces is the \emph{residual arrangement} ${R}_P \subset Y_P$ of $P$. The facet hyperplane arrangement $Y_P = \bigcup_{F \in {\cal F}(P)} H_F$ is called \emph{simple} if no $d+1$ of the hyperplanes $H_F$ meet. In particular, if $Y_P$ is simple, then so is $P$ (but not vice~versa, see Exercise \ref{ex:simplecube}).

\begin{theorem}[{\cite[Theorem 1 and Proposition 2]{kohn2020projective}}] \label{thm:kohnranestad}
Let $P \subset \mathbb{R}^d$ be a $d$-dimensional polytope. The adjoint locus $A_{P}$ contains the residual arrangement $R_P$. If $Y_P$ is simple, then $A_P$ is the unique hypersurface of degree $n-d-1$ containing $R_P$. 
\end{theorem}

\begin{example} \label{ex:respentagon}
    The residual arrangement of the pentagon $P$ from Example \ref{ex:pentagon1} consists of five points, two of which lie ``at infinity'' in $\mathbb{P}^2$. The three finite residual points are visible in Figure \ref{fig:dual+adjoint} (right). These five points are contained in a unique conic, which is the one we identified in Example \ref{ex:conicadjoint}. To gain some intuition about the fact that $R_P \subset A_P$, let us compute the residue of the form $\omega(P)$ in Example \ref{ex:pentagon1} along the line $y_2 = -1$:
    \[ {\rm Res}_{y_2 = -1} \, \omega(P) \, = \, \frac{2 - 2y_1}{(y_1+1)(-y_1)(-y_1+1)2} \, {\rm d}y_1. \]
    This must equal the canonical form of $[-1,0]$. The vanishing of the numerator at the residual point $(1,-1)$ makes sure that the spurious pole at $y_1 = 1$ cancels. 
\end{example}

\begin{remark} An analogous notion of \emph{residual arrangement} $R_{X_{\geq 0}}$ for general positive geometries $(X,X_{\geq 0})$ is proposed in \cite[Section 4.4]{lam2024invitation}. Under certain smoothness assumptions on the boundary strata, \cite[Proposition 1]{lam2024invitation} shows that $R_{X_{\geq 0}} \subseteq A_{X_{\geq 0}}$. 
\end{remark}

If $P$ is simple, the inclusion $R_P \subseteq A_P$ can be seen from Proposition \ref{prop:univadjoint}, see Exercise \ref{ex:residualarrang}.

\section{Positive geometry of polypols} \label{sec:3}

This section is about positive geometries in the plane, i.e., $X = \mathbb{P}^2$. It is mostly inspired by the paper \cite{kohn2025adjoints}. While in the previous section we worked in $d$ dimensions but allowed only linear boundary components, here we limit ourselves to $2$ dimensions but consider more curvy objects, such as the ``ear-shaped'' positive geometry in the right part of Figure \ref{fig:firstexamples}.  

\begin{definition} \label{def:polypol}
    A \emph{polypol} $(Y_\bullet, v_\bullet)$ consists of a tuple $Y_\bullet$ of $r \geq 2$ distinct irreducible plane curves $Y_1, \ldots, Y_r \subset \mathbb{P}^2$ and a tuple $v_\bullet$ of $r$ distinct points $v_{12}, v_{23}, \cdots, v_{r-1,r}, v_{r1} \in \mathbb{P}^2$ such that $v_{ij} \in (Y_i \cap Y_j) \setminus (\bigcup_{k \notin \{i,j\}} Y_k)$, $v_{ij}$ is smooth on $Y_i$ and $Y_j$ and $Y_i$ intersects $Y_j$ transversally at $v_{ij}$. The polypol $(Y_\bullet, v_\bullet)$ is \emph{rational} if $Y_1, \ldots, Y_r$ are rational plane curves.
\end{definition}

\begin{example} Recall that a curve $Y_i \subset \mathbb{P}^2$ is \emph{rational} if there exists a birational map $\mathbb{P}^1 \rightarrow Y_i$. Examples are lines, conics and singular plane cubics.
\end{example}

Notice that Definition \ref{def:polypol} is purely ``complex'': there are no assumptions on the real points of $Y_\bullet$ and $v_\bullet$. The curves $Y_i$ are called \emph{boundary curves} and the $v_{ij}$ are called \emph{vertices}. Let $Y = \bigcup_{i = 1}^r Y_i$ be the union of all boundary curves. The next proposition indicates that rational polypols provide fertile ground for finding planar positive geometries. 

\begin{proposition} \label{prop:genuszero}
Let $Y = Y_1 \cup \cdots \cup Y_r$, where $Y_i \subset \mathbb{P}^2$ are rational curves. The pair $(\mathbb{P}^2, Y)$ has genus zero in the sense of \cite{brown2025positive}.     
\end{proposition}

\begin{proof}
    This follows from Propositions 3.21 and 3.22 in \cite{brown2025positive}. 
\end{proof}

The vertices $v_{ij}$ are among the singular points of $Y$. We define the \emph{residual arrangement} $R_{\cal P}$ of the polypol ${\cal P} = (Y_\bullet, v_\bullet)$ as the set of points ${\rm Sing}(Y) \setminus \{v_{12},v_{23}, \ldots, v_{r1}\}$. We say that ${\cal P}$ is \emph{nodal} if $Y$ only has nodal singularities.

\begin{definition} \label{def:adjointpolypol}
    An \emph{adjoint curve} $A_{\cal P} \subset \mathbb{P}^2$ of a nodal polypol ${\cal P} = (Y_\bullet, v_\bullet)$ is a curve of degree $n-3$ containing the residual arrangement $R_{\cal P}$, where $n = \deg(Y_1) + \cdots + \deg(Y_r)$.
\end{definition}

In Exercise \ref{ex:adjointexists} you will show that an adjoint curve $A_{\cal P}$ always exists for a nodal polypol. Of course, Definition \ref{def:adjointpolypol} is modeled after Definition \ref{def:adjoint} and Theorem \ref{thm:kohnranestad}. The correct definition of adjoints for non-nodal polypols is more technical, see \cite[Definition 2.1]{kohn2025adjoints}.

\begin{example} \label{ex:adjpizza}
    Let $Y_1, Y_2, Y_3$ be as in Example \ref{ex:pizza} and let $v_{12} = (0,0)$, $v_{23} = (0,1)$, $v_{13} = (1,0)$ in the local coordinates $(x,y)$ used in that example. These data form a nodal polypol ${\cal P} = (Y_\bullet, v_\bullet)$. The residual arrangement consists of two points: $R_{\cal P} = \{ (-1,0), (0,-1) \}$. We have $n = 4$. There is a unique adjoint line $A_{\cal P}$. It is locally given by $1 + x + y = 0$.
\end{example}

In the nodal rational case, it makes sense to speak of \emph{the} adjoint curve of a polypol ${\cal P}$: 

\begin{theorem}[{\cite[Theorem 2.1]{kohn2025adjoints}}] \label{thm:uniqueadjoint}
    A nodal rational polypol ${\cal P} = (Y_\bullet, v_\bullet)$ has precisely one adjoint curve. That curve does not contain any of the boundary curves $Y_i$, and it does not contain any of the vertices $v_{ij}$. 
\end{theorem}

We note that, with the correct more general definition of adjoint \cite[Definition 2.1]{kohn2025adjoints}, the assumption that $P$ is \emph{nodal} can be omitted \cite[Theorem 2.1]{kohn2025adjoints}. The adjective \emph{rational} is essential in our discussion. Indeed, higher genus curves have nonzero holomorphic one-forms, which is an obstruction to uniqueness of the canonical forms on the boundary curves. In the spirit of Proposition \ref{prop:genuszero}, components of positive genus would break the genus-zero property. 

\begin{example}
    Let $Y_1$ be a smooth elliptic curve and $Y_2$ a generic line. Let $v_{12}, v_{21}$ be any of the three intersection points in $Y_1 \cap Y_2$. The remaining intersection point $p$ is the unique point in the residual arrangement $R_{\cal P}$ of ${\cal P} = (Y_\bullet, v_\bullet)$. Each line in the pencil of lines passing through $p$ is an adjoint curve of ${\cal P}$. If we replace $Y_1$ by a nodal cubic, then the residual arrangement consists of the point $p$ and the node $q$. These define a unique adjoint line $A_{\cal P}$. \end{example}

\begin{remark}
    Rational polypols and their adjoints were first introduced by Wachspress in \cite{wachspress1975rational} with the aim of using them in finite element methods for solving differential equations. The special case of polygons inspired Warren's work \cite{warren1996barycentric} on adjoints of $d$-dimensional polytopes.  
\end{remark}

The importance of Theorem \ref{thm:uniqueadjoint} in constructing positive geometries in the sense of Definition \ref{def:posgeomABL} from polypols is explained intuitively as follows. We will soon associate a semi-algebraic set $P$ to ${\cal P}$ whose algebraic boundary is $Y$. The poles of the canonical form $\omega(P)$ must lie along the $Y_i$, and the zeros of $\omega(P)$ should contain the residual arrangement in order to cancel spurious poles of the residues. This was illustrated in Example \ref{ex:respentagon}. If ``containing $R_{\cal P}$'' uniquely determines the zero locus of $\omega(P)$, then one has a unique candidate for $\omega(P)$ (up to scaling). In order to show that $(\mathbb{P}^2, X_{\geq 0})$ is a positive geometry in the sense of Definition \ref{def:posgeomABL}, we need only check the residue conditions for this candidate. 

\begin{definition} \label{def:quasireg}
    A \emph{quasi-regular polypol} is a polypol ${\cal P} = (Y_\bullet, v_\bullet)$ whose boundary curves $Y_1, \ldots, Y_r$ are real (i.e., $Y_i(\mathbb{R})$ is Zariski dense in $Y_i$), and whose vertices $v_{i-1,i} \in Y_{i-1} \cap Y_i$ are real, equipped with the following additional data: 
    \begin{enumerate}
        \item $r$ curve segments $(Y_i)_{\geq 0} \subset Y_i(\mathbb{R})$ connecting $v_{i-1,i}$ to $v_{i,i+1}$, 
        \item a semi-algebraic set $P \subset \mathbb{RP}^2$,
    \end{enumerate}
    such that $(Y_i)_{\geq 0}$ is contained in the smooth points of $Y_i$, each connected component of the interior of $P$ is simply connected, and $\partial P = \bigcup_{i=1}^r (Y_i)_{\geq 0}$.
\end{definition}

\begin{example}
    Figure \ref{fig:firstexamples} (right) shows a quasi-regular polypol with $r = 3$ and $n = 4$. The set $P$ is $X_{\geq 0}$. Figure \ref{fig:cuspexample} is \emph{not} a quasi-regular polypol, because the vertex $b$ lies in the singular locus of $Y_2$. Figure \ref{fig:nodeexample} (bottom) does \emph{not} represent a polypol, because $r = 1 < 2$. Every convex polygon $P \subset \mathbb{R}^2 \subset \mathbb{P}^2$ represents a quasi-regular polypol with $Y_\bullet = (H_F \, : \, F \in {\cal F}(P))$ given by the edge lines of $P$ and $v_\bullet = (v \, : \, v \in {\cal V}(P))$ consists of the vertices of $P$. 
\end{example}

\begin{example}
    Figure \ref{fig:quasireg} shows a (cartoon of a) quasi-regular rational polypol with $r = 4$, $n = 8$. The semi-algebraic set $P$ is shaded in green, and $\partial P$ is the bold blue curve. Notice that the curve segments $(Y_i)_{\geq 0}$ may contain singular points of $Y$, but not of $Y_i$.  
    \begin{figure}
        \centering
        \includegraphics[width=0.5\linewidth]{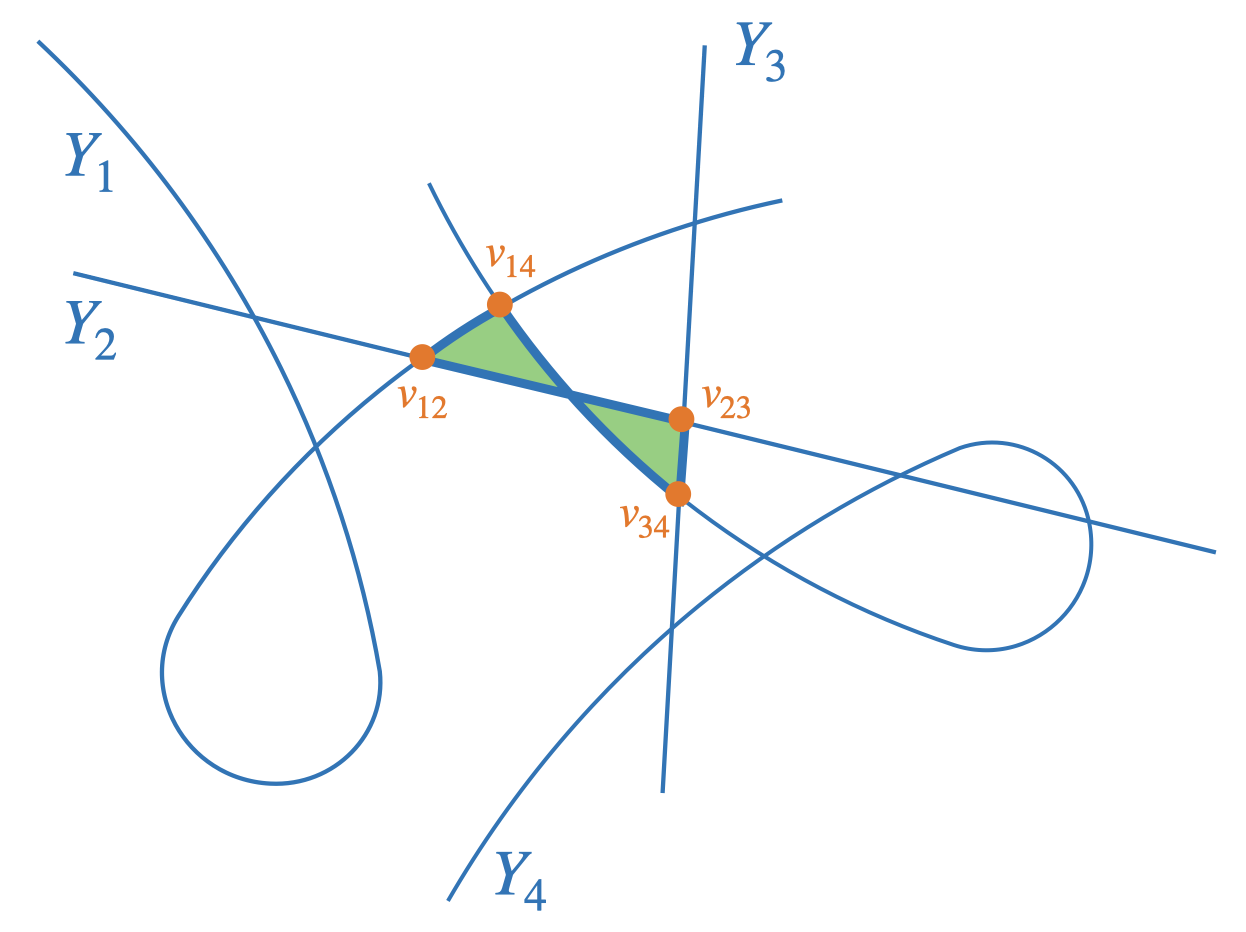}
        \caption{A quasiregular polypol with $r = 4$ and $n = 8$.}
        \label{fig:quasireg}
    \end{figure}
\end{example}

To state the next theorem, for a nodal quasi-regular rational polypol ${\cal P} = (Y_\bullet, v_\bullet)$ with semi-algebraic set $P$ we write ${\rm adj}_{\cal P}(x,y,z) \in \mathbb{C}[x,y,z]_{n-3}$ for a defining equation of $A_{\cal P}$ (Definition \ref{def:adjointpolypol}), and $f_i \in \mathbb{C}[x,y,z]_{{\rm deg}(Y_i)}$ for a defining equation of $Y_i$.

\begin{theorem}[{\cite[Theorem 2.15]{kohn2025adjoints}}] \label{thm:canformpolypol}
    A quasi-regular rational polypol ${\cal P} = (Y_\bullet, v_\bullet)$ defines a positive geometry $(\mathbb{P}^2, P)$ in the sense of Definition \ref{def:posgeomABL}. Here $P$ is the semi-algebraic subset from Definition \ref{def:quasireg}. If ${\cal P}$ is nodal, then its canonical form is, up to a constant $\alpha$, given by 
    \begin{equation} \label{eq:omegapolypol} \omega(P) \, = \, \alpha \cdot \frac{{\rm adj}_{\cal P}}{f_1 \cdot f_2 \cdot \cdots \cdot f_r} \, (x \, {\rm d}y \wedge {\rm d}z - y \, {\rm d}x \wedge {\rm d}z + z \, {\rm d}x \wedge {\rm d} y). \end{equation}
\end{theorem}

\begin{proof}
    We shall sketch the proof and refer to \cite[Theorem 2.15]{kohn2025adjoints} for details. We work in a ``sufficiently generic'' affine chart $\mathbb{C}^2 \subset \mathbb{P}^2$ which intersects each component of $A_P \cup Y$. The form $\omega(P)$ in \eqref{eq:omegapolypol} is given, up to the factor $\alpha$ and in local coordinates $(x,y)$, by 
    \[ \frac{{\rm adj}_{\cal P}(x,y)}{f_1(x,y) \cdot f_2(x,y) \cdot \cdots \cdot f_r(x,y)} \, {\rm d}x \wedge {\rm d} y.\]
    Here we abuse notation slightly by using ${\rm adj}_{\cal P}$ for a local defining equation of $A_{\cal P}$ and $f_i$ for a local equation of $Y_i$. The residue along $Y_i$ is the restriction of the following one-form to $Y_i$: 
    \[ \eta \, = \, \frac{{\rm adj}_{\cal P}}{f_1 \cdot \cdots \cdot \hat{f}_i \cdot \cdots \cdot f_r \cdot \frac{\partial f_i}{\partial y}} \, {\rm d} x. \]
    Genericity of our chart ensures that the derivative of $f_i$ with respect to $y$ is not identically zero. In our local coordinates, the normalization map $\phi_i: \mathbb{P}^1 \rightarrow Y_i$ is given by $t \mapsto (r(t)/h(t), \, s(t)/h(t))$, where $r, s, h$ are polynomials of degree $\deg(Y_i)$. The pre-image of the smooth curve segment $(Y_i)_{\geq 0}$ is an interval $[a_i, b_i] \subset \mathbb{RP}^1$, where $\phi_i(a_i) = v_{i-1,i}$ and $\phi_i(b_i) = v_{i,i+1}$. Since $\phi_i$ is an isomorphism on an open neighborhood of $\phi^{-1}((Y_i)_{\geq 0})$, we may check the residue conditions of Definition \ref{def:posgeomABL} on the pullback $\phi_i^*(\eta_{|Y_i})$. This is given by 
    \[ \phi_i^*(\eta_{|Y_i}) \, = \, \frac{{\rm adj}_{\cal P}(\phi_i(t)) \cdot (r'(t)h(t) - h'(t)r(t))}{h(t)^2 \cdot \frac{\partial f_i}{\partial y}(\phi_i(t)) \cdot \prod_{j \in \{1,\ldots, r\} \setminus \{i\}} f_j(\phi_i(t))} \, {\rm d} t\]
    When expanding this, both the numerator and denominator have a factor $h(t)^{n-1}$, where $n= \sum_{i=1}^r \deg(Y_i)$. Moreover, since $v_{i-1,i} \in Y_{i-1}$ and $v_{i,i+1} \in Y_{i+1}$, the denominator has a factor $(t-a_i)(t-b_i)$. We obtain the following simplified form: 
    \[ \phi_i^*(\eta_{|Y_i}) \, = \, \frac{p(t)}{q(t) (t-a_i)(b_i-t)} \, {\rm d}t,\]
    where $p$ and $q$ are nonzero polynomials of degree $(n-1)n_i-2$, where $n_i = \deg(Y_i)$. The fact that ${\rm adj}_{\cal P}$ does not vanish on $Y_i$ follows from Theorem \ref{thm:uniqueadjoint}. We claim that $p$ and $q$ have the same roots. By definition, the factor $h(t)^{n-3} {\rm adj}_{\cal P}(\phi_i(t))$ of $p$ vanishes at the pre-images of the nodes of $Y_i$ under $\phi_i$, and at the pre-images of the intersection points $Y_i \cap Y_j$, excluding $v_{i-1,i}, v_{i,i+1}$. This gives a list of $(n_i-1)(n_i-2) + \sum_{j\in \{1, \ldots, r\} \setminus \{ i \}} n_i n_j -2 = n_i(n-3)$ roots. Additionally, the factor $(r'(t)h(t)-h'(t)r(t))$ vanishes when the tangent line of $Y_i$ at a smooth point $\phi_i(t)$ is vertical, i.e., parallel to the $y$-axis. There are $2n_i-2$ such $t$-values, so we have found all $(n-1)n_i-2$ roots. As desired, $q(t)$ must also vanish at the intersection points $Y_i \cap Y_j$, excluding $v_{i-1,i}, v_{i,i+1}$, and its factor $h(t)^{n_i-1} \frac{\partial f_i}{\partial y}(\phi_i(t))$ vanishes at the nodes of $Y_i$ and the points with vertical tangent line. We conclude that there is a constant $\beta_i \neq 0$ such that 
    \begin{equation} \label{eq:residues} \phi_i^*(\eta_{|Y_i}) \, = \, \frac{\beta_i}{(t-a_i)(b_i-t)} \, {\rm d}t.\end{equation}
    The residues at $a_i, b_i$ are $\pm \gamma_i = \pm \beta_i (b_i - a_i)^{-1}$. We claim that the nonzero constant $\gamma_i$ is independent of $i$, so that scaling $\omega(P)$ by $\gamma_i^{-1}$ gives the desired residues $\pm 1$. For this, notice that the iterated residue at the vertex $v_{i,i+1}$ can be taken in two ways: either we first take the residue along $Y_i$, and then along $Y_{i+1}$, or we start with $Y_{i+1}$. This gives the same answer up to sign. One can check this locally by calculating the iterated residue at $(0,0)$ of the form $\frac{1}{xy} {\rm d} x \wedge {\rm d}y$. This way we see that $\gamma_i = \gamma_{i+1}$, and we continue cyclically around the polypol. 
\end{proof}

The same formula for $\omega(P)$ extends to non-nodal quasi-regular rational polypols by extending the definition of the adjoint hypersurface in the appropriate manner \cite[Theorem 2.15]{kohn2025adjoints}. The constant $\alpha$ is chosen such that the iterated residues at the vertices are $\pm 1$.

\begin{example}
    Examples \ref{ex:adjpizza} and Equation 
    \eqref{eq:canformpizza} confirm Theorem \ref{thm:canformpolypol} for the positive geometry from Example \ref{ex:pizza}. 
\end{example}

We now match Theorem \ref{thm:canformpolypol} with the Hodge-theoretic Definition \ref{def:posgeomBD} inspired by \cite{brown2025positive}.

\begin{proposition}
    Let ${\cal P} = (Y_\bullet, v_\bullet)$ be a quasi-regular polypol with semi-algebraic set $P$. The relative homology class $\sigma_P = [P] \in H_2(\mathbb{P}^2, Y)$ is a positive geometry in the sense of Definition~\ref{def:posgeomBD}. Moreover, if ${\cal P}$ is nodal, then the image of $\sigma_P$ under \eqref{eq:browndupont} equals $\omega(P)$ from~\eqref{eq:omegapolypol}.
\end{proposition}

\begin{proof}
    The fact that $\sigma_P$ is a positive geometry is a consequence of Proposition \ref{prop:genuszero}. Define $Z_i = Y_i \cap {\rm Sing}(Y)$. By \cite[Section 3.3.2]{brown2025positive}, each of the boundary pairs $(Y_i, Z_i)$ has genus zero. Hence, the relative homology class $\sigma_i = [(Y_i)_{\geq 0}] \in H_1(Y_i, Z_i)$ is a positive geometry in the sense of Definition \ref{def:posgeomBD} with canonical form $\omega_i(\sigma_i)$, where $\omega_i: H_1(Y_i, Z_i) \rightarrow \Omega^1_{\rm log}(Y_i \setminus Z_i)$ is the map from \eqref{eq:browndupont} for this one-dimensional positive geometry. To identify the form $\omega_i(\sigma_i)$, notice that under the normalization map $\phi_i: \mathbb{P}^1 \rightarrow Y_i$, the curve segment $(Y_i)_{\geq 0}$ pulls back to a line segment $[a_i,b_i] \subset \mathbb{RP}^1$. Setting $Z_i' = \phi_i^{-1}(Z_i)$, the map $\phi_i$ induces a modification $(\mathbb{P}^1 , Z_i') \rightarrow (Y_i,Z_i)$ in the sense of \cite[Definition 1.2]{brown2025positive}. By \cite[Section 2.3.1]{brown2025positive} and Example \ref{ex:linesegments}, the form $\omega_i(\sigma_i) \in \Omega^1_{\rm log}(Y_i \setminus Z_i)$ is uniquely determined by 
    \begin{equation} \label{eq:pullback} \phi_i^*(\omega_i(\sigma_i)) \, = \, \frac{b_i-a_i}{(t-a_i)(b_i-t)} \, {\rm d}t \quad \in \, \Omega^1_{\rm log}(\mathbb{P}^1 \setminus Z_i') \, \simeq \, \Omega^1_{\rm log}(Y_i \setminus Z_i) .\end{equation}
    Applying \cite[Proposition 2.14]{brown2025positive} with $X = \mathbb{P}^2$, $Y = \bigcup_{i = 1}^r Y_i$ and $Z = {\rm Sing}(Y)$, we find that $\omega(\sigma_P) \in \Omega^2_{\rm log}(\mathbb{P}^2 \setminus Y)$ is the unique logarithmic form satisfying the residue conditions ${\rm Res}_{Y_i} \omega(\sigma_P) = \omega_i( \sigma_i)$. The form $\omega(P)$ satisfies ${\rm Res}_{Y_i} \omega(P) = \omega_i( \sigma_i)$ by Equations \eqref{eq:residues} and \eqref{eq:pullback}. By \cite[Proposition 1.15]{brown2025positive}, these residue conditions imply that $\omega(P) \in \Omega^2_{\rm log}(\mathbb{P}^2 \setminus Y)$.
\end{proof}

\begin{remark} A quasi-regular polypol with semi-algebraic set $P$ is called \emph{regular} if $\bigcup_i (Y_i)_{\geq 0} \setminus \{v_{12}, \ldots, v_{r-1,r},v_{r1}\} \subset Y \setminus {\rm Sing}(Y)$, that is, the boundary of $P$ contains no singular points of $Y$ except for the vertices, and $Y \cap {\rm int}(P) = \emptyset$. A convex polygon is regular, but the polypol in Figure \ref{fig:quasireg} is not. \emph{Wachspress' conjecture} states that the adjoint curve of a regular polypol does not intersect the interior of $P$. This conjecture is wide open, see \cite[Section 3.1]{kohn2025adjoints}.
\end{remark}

\section*{Exercises}

\begin{exercise}
    Show that the canonical form of the union of two disjoint line segments in $\mathbb{RP}^1$ is the sum of the invidual canonical forms. The same is true if their interiors are disjoint.
\end{exercise}
\begin{exercise} \label{ex:quadrilateral1}
    Verify that the quadrilateral $P$ with vertices $(0,1)$, $(1,1)$, $(3,0)$, $(0,-1)$ gives a positive geometry $(\mathbb{P}^2,P)$ in the sense of Definition \ref{def:posgeomABL} with canonical form
    \[ \omega(P) \, = \, \frac{-18+x+12y}{x(y-1)(x+2y-3)(x-3y-3)} \, {\rm d} x \wedge {\rm d}y.\]
    Draw this quadrilateral and the poles and zeros of $\omega(P)$ in the plane $\mathbb{R}^2$.
\end{exercise}

\begin{exercise}
    Consider $X_{\geq 0} = \{(x,y,z) \in \mathbb{R}^3 \, : \, x \geq 0, y \geq 0, z \geq 0, x^2 + y^2 + z^2 \leq 1\}$ as a semi-algebraic subset of $X(\mathbb{R}) = \mathbb{RP}^3$. The algebraic  boundary $Y$ of $X_{\geq 0}$ is the union of three planes and a quadric in $\mathbb{P}^3$. Check that $(\mathbb{P}^3,Y)$ is a genus zero pair in the sense of \cite{brown2025positive}. Show that $(\mathbb{P}^3, X_{\geq 0})$ is a positive geometry in the sense of Definition \ref{def:posgeomABL} with canonical form 
    \[ \omega(X_{\geq 0}) \, = \, \frac{x+y+z+1}{xyz(1-x^2+y^2+z^2)} \, {\rm d}x \wedge {\rm d}y \wedge {\rm d}z. \qedhere\]
\end{exercise}

\begin{exercise} \label{ex:vertassoc}
    Compute the vertices of the three-dimensional polytope in Example \ref{ex:polytopes}. 
\end{exercise}

\begin{exercise} \label{ex:dualvolume}
    The following steps will help you match Theorems \ref{thm:canformpolytope} and \ref{thm:canformpolytopesgeneral}. 
    \begin{enumerate}
        \item For each vertex $v \in {\cal V}(P)$, define 
        \[ \sigma_v \, = \, \{ u \in (\mathbb{R}^d)^\vee \, : \, u \cdot v \leq u \cdot y \text{ for all } y \in P \}. \]
        Show that $\{ \sigma_v \, : \, v \in {\cal V}(P)\}$ partitions $\mathbb{R}^d$ into $d$-dimensional cones. More precisely, $\dim(\sigma_v) = d$, ${\rm int}(\sigma_v) \cap {\rm int}(\sigma_w) = \emptyset$ for $v,w \in {\cal V}(P)$, $v \neq w$, and $\mathbb{R}^d = \bigcup_{v \in {\cal V}(P)} \sigma_v $. (For readers familiar with polyhedral geometry, we point out that these are the maximal cones of the \emph{normal fan} of $P$.) Show that $\sigma_v$ consists of all points $u = \sum_{F \ni v} c_F \, u_F$,~$c_F \geq 0$.
        \item Assume that $P = \{ y \in \mathbb{R}^d \, : \, U \cdot y + z \geq 0 \}$ is a minimal facet representation with $z \in \mathbb{R}^n_+$, so that $0 \in {\rm int}(P)$. Show that $\sigma_v \cap P^\circ$ is the convex hull of $0 \in (\mathbb{R}^d)^\vee$ and the points $\{u_F/z_F \,:\, F \in {\cal F}(P), v \in F \}$. For this, notice that $\sigma_v \cap P^\circ = \{ u \in \sigma_v \, : \, u \cdot v \geq -1 \}$ and $u \in \sigma_v$ can be written as $u = \sum_{v \in F} c_F \, u_F$ with $c_F \geq 0$ and $u_F \cdot v = - z_F$. For our pentagon, the subdivision $P^\circ = \bigcup_{v \in {\cal V}(P)} \sigma_v \cap P^\circ$ is shown on the left side of Figure \ref{fig:dual+adjoint}.
        \item Conclude that, if $P$ is simple, then we have 
        \begin{equation} \label{eq:dualvolumefunction} {\rm vol} (P^\circ) \, = \, \sum_{v \in {\cal V}(P)} \frac{|\det U_v|}{\prod_{F \in {\cal F}(P)} z_F} \, . \end{equation}
        \item Finally, to obtain the facet description of the translated polytope $P-y$, one must replace $z \rightarrow Uy+z$. Conclude that, for $y \in {\rm int}(P)$, the function $y \mapsto {\rm vol}((P-y)^\circ)$ agrees with the canonical function from Theorem \ref{thm:canformpolytope}. \qedhere
    \end{enumerate}
\end{exercise}

\begin{exercise} \label{ex:pyramid}
    To illustrate Remark \ref{rem:subdiv}, consider the (non-simple) pyramid $P$ with vertices 
    \[ (0,0,1), \quad (1,1,0), \quad (-1,1,0), \quad (-1,-1,0), \quad (1,-1,0).\]
    Compute $\omega(P)$ by subdividing $P$ into two tetrahedra $\Delta_1, \Delta_2$ and applying Theorem \ref{thm:canformpolytope}. 
\end{exercise}

\begin{exercise}
   Let ${\cal A}_P \subset \mathbb{P}^{n-1}$ be the hypersurface defined by ${\rm Adj}_P(x) = 0$ from \eqref{eq:univadjoint} (this is the \emph{universal adjoint hypersurface}). Let $H \subset \mathbb{P}^{n-1}$ be the $d$-plane spanned by the columns of $U$ and the point $z$, interpreted as points in $\mathbb{P}^{n-1}$. Note that $H$ has coordinates $y_0, \ldots, y_d$. Show that the intersection $H \cap {\cal A}_P$ is given by $y_0 \cdot {\rm adj}_P(y) = 0$. See \cite[Figure 7]{telen2025toric} for an illustration where $P$ is a quadrilateral. For readers familiar with normal fans: Show that the adjoint ${\rm adj}_{P_{z'}}$ of a different polytope $P_{z'} = \{ y \in \mathbb{R}^d \, : \, U \cdot y + z' \}$ with the same normal fan as $P$ can be obtained analogously from the same hypersurface ${\cal A}_P$ (hence the name `universal').
\end{exercise}

\begin{exercise} \label{ex:simplecube}
    Show that the cube $P= [0,1]^3 \subset \mathbb{R}^3$ is simple, but its facet plane arrangement $Y_P$ is not. Compute the adjoint polynomial ${\rm adj}_P(y_0,y_1,y_2,y_3)$.
\end{exercise}

\begin{exercise} \label{ex:residualarrang}
    For a simple, $d$-dimensional polytope $P \subset \mathbb{R}^d$, show that $A_P \supseteq R_P \cap \mathbb{R}^d$ using Proposition \ref{prop:univadjoint}. Hint: $L_S \cap P = \emptyset$ if and only if $L_S$ contains no vertex of $P$. 
\end{exercise}

\begin{exercise} \label{ex:adjointexists}
    Show that an adjoint curve $A_{\cal P}$ of a nodal polypol ${\cal P}$ always exists. Hint: an irreducible plane curve of degree $k$ has at most $\frac{1}{2}(k-1)(k-2)$ singularities. 
\end{exercise}

\begin{exercise}
    We consider the quasi-regular polypol defined by the following data. The boundary curves are the four lines $Y_1, Y_2, Y_3$ and $Y_4$ given locally by $y = 0, 2y = -x+2, y = -2x+2$ and $x=0$ respectively. The vertices are as shown in Figure \ref{fig:nonconvexquad}. The semi-algebraic set $P$ is shaded in green. Compute the canonical form $\omega(P)$ using Theorem \ref{thm:canformpolypol}. 
    \begin{figure}
        \centering
        \includegraphics[width=0.32\linewidth]{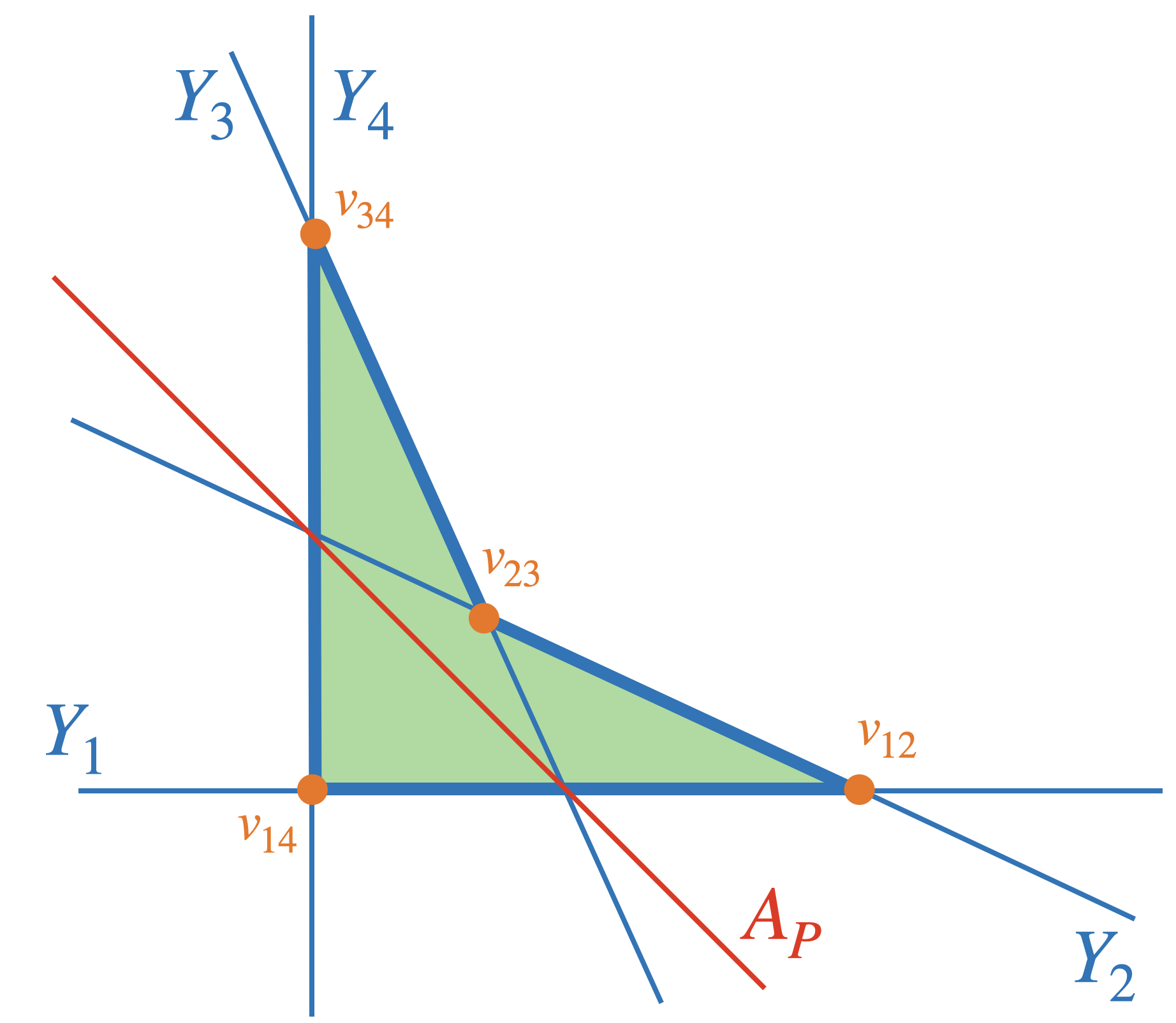}
        \caption{A non-convex quadrilateral is a quasi-regular polypol.}
        \label{fig:nonconvexquad}
    \end{figure}
\end{exercise}

\section*{Acknowledgements}
I am grateful to Jonathan Boretsky, Matteo Parisi and Lauren Williams for organizing the summer school and for including these lectures in the program. I want to thank Cl\'ement Dupont and Rainer Sinn for useful comments on a previous version of this manuscript.

\footnotesize
\bibliographystyle{abbrv}
\bibliography{references}

\noindent{\bf Author's address:}
\medskip
\noindent Simon Telen, MPI-MiS Leipzig
\hfill {\tt simon.telen@mis.mpg.de}

\end{document}